\documentclass[10pt,reqno]{article}     

\usepackage[margin=1in,asymmetric]{geometry}

\usepackage{fancyhdr}
\usepackage{amsmath,amsthm,amssymb,amscd,fancybox,ifthen,float,epsfig,subfigure}
\usepackage{graphicx}
\usepackage{setspace}
\usepackage[all]{xy}
\usepackage{comment}

\usepackage{mathpazo}

\numberwithin{equation}{section}

\newtheorem{theorem}{Theorem}[section]
\newtheorem{lemma}[theorem]{Lemma}
\newtheorem{corollary}[theorem]{Corollary}
\newtheorem{proposition}[theorem]{Proposition}
\theoremstyle{definition}

\theoremstyle{remark}
\newtheorem{remark}[theorem]{Remark}

\floatstyle{plain}
\restylefloat{figure}

\newcommand\In{\operatorname{in}}
\newcommand\Out{\operatorname{out}}

\newcommand\Sc{\operatorname{sc}}

\newcommand\cT{\mathcal T}
\newcommand\cM{\mathcal M}
\newcommand\cN{\mathcal N}
\newcommand\cH{\mathcal H}

\newcommand\cJ{\mathcal J}

\newcommand\cQ{\mathcal Q}

\newcommand\cG{\mathcal G}
\newcommand\cB{\mathcal B}

\newcommand\cP{\mathcal P}

\newcommand\balpha{\boldsymbol \alpha}
\newcommand\tbalpha{\widetilde{\boldsymbol \alpha}}

\newcommand\btheta{\boldsymbol \theta}
\newcommand\hbtheta{\widehat{\boldsymbol \theta}}
\newcommand\hbphi{\widehat{\boldsymbol \phi}}

\newcommand\hbn{\widehat{\boldsymbol n}}

\newcommand\bEta{\boldsymbol \eta}

\newcommand\tmu{\widetilde{\mu}}
\newcommand\tomega{\widetilde{\omega}}

\newcommand\bBeta{\boldsymbol \beta}
\newcommand\bj{\boldsymbol j}

\newcommand\bm{\boldsymbol m}

\newcommand\bn{\boldsymbol n}

\newcommand\bx{\boldsymbol x}
\newcommand\by{\boldsymbol y}
\newcommand\hbx{\hat{\boldsymbol x}}

\newcommand\ba{\boldsymbol a}
\newcommand\bB{\boldsymbol B}
\newcommand\bb{\boldsymbol b}

\newcommand\bH{\boldsymbol H}

\newcommand\bE{\boldsymbol E}

\newcommand\cE{\mathcal E}
\newcommand\cW{\mathcal W}

\newcommand\bxi{\boldsymbol \xi}

\newcommand{\THME[1]}{THME(#1)}

\newcommand\Ker{\operatorname{ker}}

\newcommand\cC{\mathcal{C}}
\newcommand\cL{\mathcal{L}}
\newcommand\cS{\mathcal{S}}

\newcommand\cD{\mathcal{D}}

\newcommand\bomega{\boldsymbol{\omega}}

\renewcommand\Im{\operatorname{Im}}

\newcommand\bbC{\mathbb C}

\newcommand\bbR{\mathbb R}
\newcommand\bbZ{\mathbb Z}

\newcommand\pa{\partial}

\newcommand\restrictedto{\upharpoonright}

\newcommand\CI{{\mathcal C}^{\infty}}

\newcommand\Id{\operatorname{Id}}

\DeclareMathOperator{\dR}{dR}

\begin{document}

\title{Debye Sources, Beltrami Fields, and a Complex Structure on Maxwell
  Fields}

\author{Charles L. Epstein\footnote{Depts. of Mathematics and
    Radiology, University of Pennsylvania, 209 South 33rd Street,
    Philadelphia, PA 19104. E-mail: cle@math.upenn.edu. Research of
    C.L.~Epstein partially supported by NSF grants DMS09-35165 and
    DMS12-05851.}, \, 
Leslie Greengard\footnote{Courant Institute, New York University, 251
    Mercer Street, New York, NY 10012.  E-mail:
    greengard@cims.nyu.edu. Research of L.~Greengard partially
    supported by the U.S. Department of Energy under contract
    DEFG0288ER25053 and by the Air Force Office of Scientific Research
    under NSSEFF Program Award FA9550-10-1-0180.}, \, and Michael
  O'Neil\footnote{Courant Institute, New York University, 251 Mercer
    Street, New York, NY 10012.  E-mail: oneil@cims.nyu.edu. Research
    of M.~O'Neil partially supported by the Air Force Office of
    Scientific Research under NSSEFF Program Award FA9550-10-1-0180.
}}

\maketitle

\begin{abstract} The Debye source representation for solutions to the time
    harmonic Maxwell equations is extended to bounded domains with
    finitely many smooth boundary components. A strong uniqueness
    result is proved for this representation. Natural complex
    structures are identified on the vector spaces of time-harmonic
    Maxwell fields. It is shown that in terms of Debye source data,
    these complex structures are uniformized, that is, represented by
    a fixed linear map on a fixed vector space, independent of the
    frequency. This complex structure relates time-harmonic Maxwell
    fields to constant-$k$ Beltrami fields, i.e. solutions of the
    equation
    \begin{equation*}
      \nabla\times\bE=k\bE.
    \end{equation*}
    A family of self-adjoint boundary conditions are defined for the
    Beltrami operator. This leads to a proof of the existence of
    zero-flux, constant-$k$, force-free Beltrami fields for
    any bounded region in $\bbR^3$, as well as a constructive method
    to find them. The family of self-adjoint boundary value problems
    defines a new spectral invariant for bounded domains in
    $\bbR^3.$\\

\end{abstract}

\onehalfspacing

\section{Introduction}
\label{sec:uniq}
In several previous papers~\cite{EpGr,EpGr2}, we introduced a new
representation for the time-harmonic Maxwell equations in $\bbR^3$
based on two scalar densities defined on the surface $\partial D$ of a
smooth bounded region $D$, which we refer to as generalized Debye
sources. Recall that a pair of vector fields $(\bE,\bH)$ defined in an
open subset of $\bbR^3$, satisfies the \THME[$k$] if
\begin{equation}
  \nabla\times\bE= ik \bH\quad \text{ and } \quad \nabla\times\bH= -ik \bE.
\end{equation}
If $k\neq 0$, then this implies the divergence equations
\begin{equation}
  \nabla\cdot\bE \, = \, \nabla\cdot\bH \, = \,0.
\end{equation}
In the discussion that follows, unless otherwise noted, $D$ 
denotes a bounded region in $\mathbb R^3$ with smooth boundary
$\partial D$. When considering  an exterior problem, the unbounded
domain with smooth boundary $\Gamma$  is often referred to as $\Omega,$ though
$\Omega$ is sometimes used to refer to bounded components of $\Gamma^c$ as
well.

As in~\cite{EpGr,EpGr2}, we use exterior forms to
represent Maxwell fields. We use a 1-form $\bxi$ to represent the
electric field and 2-form $\bEta$ for the magnetic field. Faraday's
law and Ampere's law (the curl equations) then take the form:
\begin{equation}
  \label{eq:ME1}
  d\bxi=ik\bEta, \qquad d^*\bEta=-ik\bxi,
\end{equation}
where 
$$k \in \mathbb C^+=\{z:\Im z\geq 0\}.$$ 
For $k\neq 0$, these equations imply the divergence equations, which take the
form:
\begin{equation}
  \label{eq:ME2}
  d^*\bxi=0, \qquad d\bEta=0.
\end{equation}
We call the system of equations~\eqref{eq:ME1} and~\eqref{eq:ME2} the
\THME[$k$].  Together, \eqref{eq:ME1} and~\eqref{eq:ME2} form an
elliptic system for the pair $(\bxi,\bEta)$.

In our earlier work the emphasis was on applications to scattering
theory. This previous analysis was performed largely in an exterior domain
$\Omega$ where the solution takes the form
\begin{equation}
  (\bxi,\bEta) = (\bxi^{\Sc},\bEta^{\Sc}) +(\bxi^{\In},\bEta^{\In}).
\end{equation}
The components $(\bxi^{\In},\bEta^{\In})$ are called the incoming
field; the scattered field, $(\bxi^{\Sc},\bEta^{\Sc})$, is assumed to
satisfy the outgoing radiation condition in $\Omega$. For the electric
field, this reads:
\begin{equation}
  \label{eq:radcond}
  i_{\hbx}d\bxi-ik\bxi=O\left(\frac{1}{|x|^2}\right),\quad
\bxi=O\left(\frac{1}{|x|}\right),
\end{equation}
where $\hbx=\frac{\bx}{\|\bx\|}$. The same condition is satisfied by
$\star_3\bEta.$ We use the notation $\star_j$ for the Hodge star operator
acting on forms defined on a $j$-dimensional oriented, Riemannian manifold. In
the present context $\star_3$ is the Hodge star operator on $\bbR^3$ with the
Euclidean metric, and standard orientation. It is a classical result that if
$(\bxi,\bEta)$ solves the \THME[$k$] for $k\neq 0$, and one component is
outgoing, then so is the other. When $k=0$ the equations for $\bxi$ and $\bEta$
decouple; the divergence equations,~\eqref{eq:ME2}, are no longer a consequence
of the curl equations, but are nonetheless assumed to hold. 

In this paper our emphasis shifts to an application of the Debye
source representation in bounded domains $D\subset\bbR^3$. We assume
that $\partial D$ is smooth, but allow it to have multiple connected
components.  The Debye source representation uses non-physical scalar
sources along with harmonic 1-forms on $\partial D$ to represent
the space of solutions to~\eqref{eq:ME1} in a manner that is
insensitive to the choice of wave number $k \in \mathbb C^+$.
The Debye sources are used to define a pair of
pseudocurrents $\bj$ and $\bm$ on $\partial D$, which we call {\em Debye
currents}. For an appropriate relationship between these
pseudocurrents, e.g.~\eqref{eqn3.21.001}, we show that the Debye
source representation is injective. This
representation leads to Fredholm equations of the second kind for solutions
to~\eqref{eq:ME1} with specified tangential $\bxi$ or $\bEta$
components along $\partial D$; this implies that the representation
with the restricted source data is also surjective.

We then turn our attention to a very interesting feature of the space
of solutions to the time-harmonic Maxwell system: it has a complex
structure. If $(\bxi,\bEta)$ solves~\eqref{eq:ME1} in a domain $D$,
then so does 
\begin{equation}
  \cJ(\bxi,\bEta)=(-\star_3\bEta,\star_3\bxi).
\end{equation}
  In terms
of the more traditional vector field notation, this is just the
statement that if $(\bE,\bH)$ solves the \THME[$k$], then so does
$(-\bH,\bE)$.

Since $\star_3^2=\Id$, we see that
\begin{equation}
  \cJ^2(\bxi,\bEta)=-(\bxi,\bEta).
\end{equation}
In other words, $\cJ$ defines a complex structure on the vector space,
$\cM_k(D)$, of solutions to \THME[$k$] in $D$. Its eigenvalues are $\pm i,$ and
this immediately
implies that this space of solutions is a direct sum of the two eigenspaces,
one in which $\star_3\bEta=-i\bxi$ and another in which
$\star_3\bEta=i\bxi$.  These subspaces are denoted by
$\cM_k^{(1,0)}(D)$ and $\cM_k^{(0,1)}(D)$, respectively. In these
subspaces the electric field satisfies the equations
\begin{equation}
  \star_3d\bxi= k\bxi\quad \text{ and } \quad \star_3d\bxi=-k\bxi, 
\end{equation}
which are called Beltrami equations.  The operator
$\star_3d$ is the exterior form representation of the vector 
curl operator.  In the classical vector notation, these equations are
\begin{equation}
  \nabla\times\bE=k\bE \quad \text{ and } \quad \nabla\times\bE=-k\bE.
\end{equation}

For applications in fluid mechanics and plasma physics it is
especially interesting to find solutions with vanishing normal
components \cite{bauer,hudson,hudson2,kress2,marsh}. These are often
referred to as \emph{constant-$k$, force-free fields}.  A relationship
of this sort was used in~\cite{PicardCurlED} to study the Beltrami
equation, though without explicitly defining the complex structure on
$\cM_k(D)$.

If $(\bj,\bm)$ are the Debye currents that represent a solution
$(\bxi,\bEta)$, then $(-\bm,\bj)$ are the currents that represent
$\cJ(\bxi,\bEta)$. This shows that the Debye representation provides a
uniformization of this complex structure on $\cM_k(D)$: for any value
of $k$, the structure $\cJ$ is represented by a \emph{fixed} linear
transformation on a \emph{fixed} vector space. This observation leads
to effective numerical methods for solving the Beltrami equation in
either a bounded or unbounded domain.

As in our earlier work on solving the Maxwell system, we obtain Fredholm
equations of second kind for the normal component of a Beltrami field. If
$\partial D$ is not simply connected, then these integral equations need to be
augmented with $p$ algebraic conditions, where $p$ is the genus of $\partial
D$. There are many possible choices for these conditions, which are effectively
parametrized by the Lagrangian subspaces, $\Lambda^1_H(D),$ relative to the
canonical wedge product pairing, of the de Rham cohomology group
$H^1_{\dR}(\partial D)$. Each $\lambda\in\Lambda^1_H(D)$ defines a self-adjoint
boundary value problem for the curl-operator with a real spectrum
$\sigma(\lambda,D).$ 

The intersection of all these spectra,
\begin{equation}
  E_{\cB}=\bigcap_{\lambda\in\Lambda_H(D)}\sigma(\lambda,D) ,
\end{equation}
is an invariant of the embedding of $D$ into $\bbR^3.$ A priori one might
expect this to be a finite set, but for round balls and tori of revolution we
show that $|E_{\cB}|$ is infinite. We also show that, in all cases, $k\in E_{\cB}$ if and
only if there is a 1-form $\bxi$ defined in $D$ with
\begin{equation}
  d^*\bxi=0,\quad \star_3d\bxi=k\bxi,\quad d_{\pa D}\bxi_{t}=0,
\end{equation}
with $[\bxi_t]$ representing the trivial class in $H^1_{\dR}(\pa D).$ Here, and
in the sequel: 
$$\bxi_t\overset{d}{=}\bxi\restrictedto_{T\pa D}.$$

These conditions have been defined and analyzed by several authors, see for
example~\cite{HKTSACurlBD,PicardCurlBD,PicardCurlED}. We provide a new, and
somewhat simpler proof, of the self-adjointness of these unbounded
operators. Unlike in the existing literature, we provide a straightforward way
to reduce the solution of the Beltrami equation to well-conditioned, integral
equations on the boundary. Using an alternative integral representation, Kress
worked out something similar in the case of a torus, see~\cite{kress2}.
Additionally, the strong uniqueness result for the Debye source representation
also allows us to use these boundary equations to find both the spectrum and
the eigenvectors of these self-adjoint extensions of the curl operator. We give
numerical results for the cases of  the unit ball and torus-of-revolution. We close by
briefly considering the analogous self-adjoint boundary conditions for the
time-harmonic Maxwell equations in a bounded domain.

\section{Debye Sources and Potentials}\label{s.potentials}

In this section, we quickly review the representation of solutions to
the \THME[$k$] in terms of both potentials and anti-potentials.  For
the moment, we will assume that we are working in \emph{either} a
bounded \emph{or} an unbounded domain $\Omega$ with smooth boundary, connected
$\Gamma$.  The final assumption is just for ease of exposition; later in the
paper we consider regions whose boundaries have several components.

For $k \in \mathbb C^+$, as in~\cite{EpGr,EpGr2}, we
represent the solution to the \THME[$k$] by setting:
\begin{equation}
\begin{split}
\bxi=ik\balpha-d\phi-d^*\btheta &=ik\balpha-d\phi-\star_3d\star_3\btheta,\\
\bEta=ik\btheta-d^*\Psi+d\balpha&=ik\btheta+\star_3d\psi+d\balpha,
\end{split}
\label{eqn29}
\end{equation}
where $\phi$ is a scalar function, $\balpha$ a 1-form, $\btheta$ a
2-form, and $\Psi=\psi dV$, a 3-form; $\balpha$ is the usual vector
potential and $\phi$ the corresponding scalar potential, while
$\btheta$ is the vector anti-potential and $\psi$ the corresponding
scalar anti-potential.

We assume that all of the potentials solve the Helmholtz equation,
$\Delta\bBeta+k^2\bBeta=0,$ in the correct form degree. Here $\Delta=-(d^*d+dd^*)$ denotes the
(negative) Laplace operator.  In order for
$(\bxi, \bEta)$ to satisfy Maxwell's equations in $\Gamma^c$:
\begin{equation}
  \label{eq:1.29.8.1}
  d\bxi=ik\bEta, \qquad d^*\bEta=-ik\bxi,
\end{equation}
it suffices to check that (in the Lorenz gauge)
\begin{equation}
d^*\balpha=-ik \phi, \qquad d\btheta=-ik\Psi.
\label{pteq3}
\end{equation}

We let $g_k(x,y)$ denote the fundamental solution for the scalar
Helmholtz equation with Helmholtz parameter $k \in \mathbb C^+,$ which
satisfies the outgoing Sommerfeld radiation condition:
\[
g_k(x,y) = \frac{e^{i k |x-y|}}{4 \pi |x-y|}.
\]
All of the potentials can be expressed in terms of a pair of 1-forms $\bj, \bm$
defined on $\Gamma$, which define electric and magnetic currents. As $\Gamma$
is embedded in $\bbR^3$, these 1-forms can be expressed in terms of the ambient
basis from $\bbR^3$, e.g.,
\begin{equation}
\bj=j_1(x)dx_1+j_2(x)dx_2+j_3(x)dx_3.
\end{equation}
We normalize both with the requirement
\begin{equation}
i_{\bn}\bj=i_{\bn}\bm\equiv 0,
\label{1frmnrm}
\end{equation}
which is analogous to the classical requirement that the current be
tangent to the boundary. We call forms satisfying this condition
\emph{tangential 1-forms} on $\Gamma.$

It is well-known \cite{Jackson,Papas} that conditions \eqref{pteq3}
are satisfied if
\begin{equation}
\begin{aligned}
\balpha&=\int\limits_{\Gamma}g_k(x,y)[\bj(y)\cdot d\bx]dA(y), &\qquad
\phi&=\frac{1}{ik}\int\limits_{\Gamma} g_k(x,y) d_{\Gamma}
\star_2\bj(y), \\ \btheta&=\star_3\int\limits_{\Gamma}g_k(x,y)
     [\bm(y)\cdot d\bx]dA(y), &\qquad
     \Psi&=-\frac{dV_x}{ik}\int\limits_{\Gamma} g_k(x,y)
     d_{\Gamma}\star_2\bm(y).
\label{srfint2}
\end{aligned}
\end{equation}

For most applications of the Debye source representation we do
\emph{not} use the currents $\bj$ and $\bm$ as the \emph{fundamental}
parameters, though in the present context we will sometimes find this
to be useful. In~\cite{EpGr}, we introduced the notion of generalized
Debye sources, $r$, $q$, which are scalar functions defined on $\Gamma$
by the differential equations:
\begin{equation}
\frac{1}{ik}d_{\Gamma}\star_2\bj=rdA,\quad
\frac{1}{ik}d_{\Gamma}\star_2\bm=-qdA,
\label{eqn53}
\end{equation}
known as \emph{consistency conditions}.
From this definition, we see that $rdA$ and $qdA$ are exact and hence
their mean values vanish on $\Gamma$,
\begin{equation}\label{eqn3.8.004}
\int\limits_{\Gamma}rdA=\int\limits_{\Gamma}qdA=0.
\end{equation}
This is {\bf necessary} for the conditions in~\eqref{eqn53} to hold,
and thus, for $(\bxi,\bEta)$ to satisfy the Maxwell equations. In
terms of the generalized Debye sources:
\begin{equation}
\begin{split}
\phi&=\int\limits_{\Gamma} g_k(x,y)r(y)dA(y), \\
\Psi&= dV_x \int\limits_{\Gamma} g_k(x,y) q(y)dA(y).
\end{split}
\label{srfint3}
\end{equation}
It should be emphasized that any data $(r,q,\bj,\bm)$ that
satisfy~\eqref{eqn53} and~\eqref{eqn3.8.004} define a solution to the
\THME[k] in $\Gamma^c$. We call the pairs $(r,q)$ scalar \emph{Debye source
  data}. This flexibility allows us to easily construct elements of
$\cM_k(\Omega)$, which thereby leads to an efficient method for
finding Beltrami fields.

\subsection{Boundary equations}
Following the convention in~\cite{EpGr,EpGr2}, we use $\star_2\bxi^{\pm}_t$ and
$\star_2([\star_3\bEta^{\pm}]_t)$, which correspond to $\bn\times\bE^{\pm}$ and
$\bn\times\bH^{\pm}$, respectively, to represent the tangential components. Here $+$
refers to the limit from the unbounded component of $\Gamma^c$ and $-$ the
limit from the bounded component. Note that
$i_{\bn}\bEta=-\star_2[\star_3\bEta]_t$. The scalar functions
$i_{\bn}\bxi^{\pm}$ and $i_{\bn}[\star_3\bEta^{\pm}]$, which equal
$\bn\cdot\bE^{\pm}$, and $\bn\cdot\bH^{\pm}$, represent the normal components. These
limiting values satisfy the integral equations:
\begin{equation}
\begin{split} 
\left(\begin{matrix}\star_2\bxi^{\pm}_t\\
\star_2([\star_3\bEta^{\pm}]_t)\end{matrix}\right) &=  
\frac{1}{2}\left(\begin{matrix}\mp\bm\\\pm\bj\end{matrix}\right)
+\left(\begin{matrix}-K_1& 0 & ikK_{2,t}& -K_4\\
0 & K_1 & K_4 & ik K_{2,t}\end{matrix}\right)
\left(\begin{matrix}r\\q\\\bj\\\bm\end{matrix}\right)\\
&\overset{d}=
  \left(\begin{matrix}\cT^{\pm}_{\bxi}\\
\cT^{\pm}_{\bEta}\end{matrix}\right)(r,q,\bj,\bm),
\label{tngtds0}
\end{split}
\end{equation}
and
\begin{equation}
\begin{split} 
  \left(\begin{matrix}i_{\bn}\bxi^{\pm}\\
i_{\bn}[\star_3\bEta^{\pm}]\end{matrix}\right)&= 
\frac{1}{2}\left(\begin{matrix}\pm r\\\mp q\end{matrix}\right)
+\left(\begin{matrix}-K_0& 0 & ikK_{2,n}& -K_3\\
 0 & K_0 & K_3 & ik K_{2,n}\end{matrix}\right)
\left(\begin{matrix}r\\q\\\bj\\\bm\end{matrix}\right)\\
&\overset{d}= \left(\begin{matrix}\cN^{\pm}_{\bxi}\\
\cN^{\pm}_{\bEta}\end{matrix}\right)(r,q,\bj,\bm).
\label{nrmds0}
\end{split}
\end{equation}
The equations in~\eqref{tngtds0} and~\eqref{nrmds0} correct  sign errors in~\cite{EpGr2}.
The integral operators $K_0, K_1, K_{2,t}, K_{2,n}, K_3$ and $K_4$ are defined
in Appendix~\ref{bdryops}.

For non-zero wave numbers, we can use~\eqref{eqn53} to rewrite these boundary
equations in terms of the currents $\bj$ and $\bm$ alone. For example:
\begin{equation}
\left(\begin{matrix}\star_2\bxi^{\pm}_t\\
\star_2([\star_3\bEta^{\pm}]_t)\end{matrix}\right) = 
\frac{1}{2}\left(\begin{matrix}\mp\bm\\\pm\bj\end{matrix}\right)
+\left(\begin{matrix}-K_1& 0 & ikK_{2,t}& -K_4\\
0 & K_1 & K_4 & ik K_{2,t}\end{matrix}\right)
\left(\begin{matrix}-d_{\Gamma}^*\bj/ik\\d_{\Gamma}^*\bm/ik\\\bj\\\bm\end{matrix}\right)
\label{tngtds0.01}
\end{equation}
where $d_{\Gamma}^*\balpha=-\star_2d_{\Gamma}\star_2\balpha$, on a
1-form. Through this formulation, any pair of 1-forms
$(\bj,\bm)$ defines a solution to the \THME[$k$] in $\Gamma^c.$ 
In general these fields might vanish in one component or the other. By
imposing relations between these 1-forms we can obtain injective, and
therefore surjective representations. This is described in the next
section.

\section{Uniqueness Theorems}

We now turn our attention to Maxwell fields in \emph{interior} domains.  As is
always the case with representations yielding Fredholm equations of index zero,
proving surjectivity is reduced to proving injectivity. In~\cite{EpGr}
and~\cite{EpGr2} this is established for outgoing solutions in an exterior
domain $\Omega$, possibly with multiple boundary components. In those works, we
produced a Fredholm system of second kind for the solution to the \emph{perfect
  electric conductor} problem, wherein $\bxi_t$ is specified on $\Gamma$, the
boundary of $\Omega$. By showing that this system of equations has trivial
null-space it follows from the Fredholm alternative that the representation is
surjective.

There is an analogous system of Fredholm equations of second kind for
the solution to the interior problem for equations~\eqref{eq:ME1} in
$D$ with either $\bxi_t$ or $[\star_3\bEta]_t$ specified on $\partial
D$. The interior problem is different from the exterior in that there
exists a sequence of real frequencies $\{k_j\}$ for which there are
non-trivial solutions to~\eqref{eq:ME1} that have a vanishing tangent
component of the $\bxi$-field. On the other hand we can use the
following lemma to obtain the desired result.
\begin{lemma}\label{lem3.2.002}
Let $k\in \mathbb C^+$, and $D$ be a bounded domain with smooth
connected boundary $\partial D$. If $\bm=-\star_2\bj$, then with
  \begin{equation}\label{eqn3.13.003}
    \bj=ik(d_{\Gamma}R_0r-\star_2 d_{\Gamma}R_0q)+\bj_H,
  \end{equation}
 the map
  \begin{equation}\label{eqn3.2.006}
    (r,q,\bj_H)\mapsto  \left(\begin{matrix}\cT^{-}_{\bxi}\\
\cT^{-}_{\bEta}\end{matrix}\right)(r,q,\bj,-\star_2\bj)
  \end{equation}
is injective. Here $R_0$ is the inverse of the (negative) Laplacian
$\Delta_{\Gamma}$ on scalar functions of mean zero and $\bj_H$
is a tangential harmonic 1-form on $\pa D.$
\end{lemma}
\begin{proof} We first observe that if $(r,q,\bj_H)$ is in the null-space of
  the map in~\eqref{eqn3.2.006}, then the tangent components,
  $(\bxi^{-}_t,\bEta^{-}_t),$ both vanish. Theorem 4.1
  in~\cite{ColtonKress}, or Theorem 5.5.1 in~\cite{Nedelec}  show that
  $(\bxi^{-},\bEta^{-})$ is determined in $D$ by the pair of tangent
  components along $\partial D$, and therefore $(\bxi^{-},\bEta^{-})$ vanishes
  identically in $D.$

We first treat the case $k\neq 0$. Assume that we have data
$(r,q,\bj_H)$ so that, with $\bm=-\star_2\bj$, the solution
$(\bxi^-,\bEta^-)$ in $D$ defined by~\eqref{eqn29} is zero.  The jump
conditions then imply that
\begin{equation}\label{eqn3.14.03}
  \bxi_t^+=\bj\quad \text{ and } \quad i_{\bn}\bEta^+=-\bj.
\end{equation}
In the exterior $D^c$, the outgoing condition and equation (6.14)
from~\cite{EpGr} state that:
\begin{multline}\label{eqn3.15.03}
  \lim_{R\to\infty}\bigg(2\Im(k)\int\limits_{D^c_R}[\|d\bxi^+\|^2+|k|^2\|\bxi\|^2]dV +\int\limits_{S_R}[\|i_{\hbx}d\bxi^+\|^2+|k|^2\|\bxi^+\|^2]dA\bigg)=\\ -2\Im\left[k\int\limits_{\partial
      D}\bxi^+\wedge\star_2i_{\bn}\overline{d\bxi^+}\right].
\end{multline}
Here $D^c_R=D^c \cap B_R$.  Using the PDE and~\eqref{eqn3.14.03} we see
that
\begin{equation}
  -2\Im\left[k\int\limits_{\Gamma}\bxi^+\wedge\star_2i_{\bn}\overline{d\bxi^+}\right]=
-2|k|^2\int\limits_{\Gamma}\bj\wedge\star_2\overline{\bj}.
\end{equation}
As the left hand side of~\eqref{eqn3.15.03} is clearly non-negative,
this shows that $\bj=0$. As $\bj$ satisfies~\eqref{eqn3.13.003}, the
fact that $r$ and $q$ have mean zero then completes the proof of the
lemma.

The $k=0$ case is quite similar to the argument in the proof of
Theorem 5.2 in~\cite{EpGr2}. Suppose that we have data $(r,q,\bj_H)$
so that the vector sources satisfy $\bm=-\star_2\bj$ and the harmonic
fields $(\bxi^-,\bEta^-)=0$. We note that at zero frequency,
$\bj=\bj_H$. The jump conditions therefore show that $\bxi^+_t=\bj_H$
and $[\star_3\bEta^+]_t=\star_2\bj_H$. On the other hand, $\bxi^+$ and
$\star_3\bEta^+$ represent cohomology classes in $H^1_{\dR}(D^c)$, and
therefore $\bxi^+_t$ and $[\star_3\bEta^+]_t$ belong to the
complexification of a real Lagrangian subspace of $H^1_{\dR}(\pa D)$ with respect to the
wedge product pairing, hence
\begin{equation}
  \int\limits_{\partial D}\bj_H\wedge\star_2\overline{\bj_H}=0.
\end{equation}
The fact that $\bxi_-$ and $\bEta_-$ both vanish, along
with~\eqref{nrmds0}, implies that
\begin{equation}\label{eqn3.19.001}
  \frac{r}{2}+K_0r=0\quad \text{ and } \quad \frac{q}{2}+K_0q=0.
\end{equation}
As $r$ and $q$ are assumed to have mean zero, it is classical
(see~\cite{ColtonKress}) that~\eqref{eqn3.19.001} implies that $r=q=0$.
\end{proof}

For later applications the following surjectivity result is very useful.
\begin{proposition}\label{prop3.1} 
  If $D$ is a bounded connected region in $\bbR^3$ with smooth connected
  boundary $\partial D$, then for $k\in \mathbb C^+$, all solutions to
  \THME[$k$] in $D$ are represented by Debye source data $(r,q,\bj_H).$
  Similarly, for an unbounded region, all outgoing solutions are represented by
  such data.
\end{proposition}
\begin{proof} 
  The unbounded case is done in~\cite{EpGr}.  To prove this in the
  bounded case for non-zero frequencies, we use a hybrid system of
  Fredholm equations of second kind analogous to equations (5.5-5.7)
  in~\cite{EpGr2}. Indeed, we use precisely the same set-up with the
  single change that $\bm=-\star_2\bj$ instead of
  $\bm=\star_2\bj$. We write this schematically as
\begin{equation}\label{eqn3.17.001}
  \cQ^-(k)\left(\begin{matrix} r\\q\\\bj_H\end{matrix}\right)\overset{d}{=}
\left(\begin{matrix} G_0 d_{\Gamma}^*\cT^{-}_{\bxi}(k;r,q,\bj_H)\\\cN^{-}_{\bEta}(k;r,q,\bj_H)\end{matrix}\right)=\left(\begin{matrix} f\\g\end{matrix}\right).
\end{equation}
If $\pa D$ has genus $p>0,$ then we append the cohomological
conditions:
\begin{equation}\label{eqn3.18.001}
\begin{split}
  \frac{1}{ik}\int\limits_{A_j}\star_2\cT^{-}_{\bxi}(k)&=
  -a_j ,\\
  \int\limits_{B_j}\star_2\cT^{-}_{\bxi}(k)&=-b_j,
\end{split}
\end{equation}
for $j = 1, \ldots p$. Here
$\{a_1,\dots,a_p;b_1,\dots,b_p\}\in\bbC^{2p}.$ The $A$-cycles $\{
A_j\}$ bound chains, $\{S_j\}$, contained in $D$ and the $B$-cycles
$\{ B_j\}$ bound chains, $\{T_j\}$, contained in $D^c$. Together
$\{A_j,B_j\}$ are a basis for $H_1(\pa D),$ which can
be taken to be formal sums of smooth, simple closed curves on
$\Gamma.$ As these conditions entail restriction to codimension 1
submanifolds, they are given by bounded functionals provided that
$\cT^{-}_{\bxi}(k)\in W^{s,2}(\Gamma)$ for an $s> 1/2.$ 
For bounded solutions to the \THME[$k$], as $k$ tends
to zero
\begin{equation}
\int\limits_{A_j}\star_2\cT^{-}_{\bxi}(k) \approx \mathcal O(k).
\end{equation}
For small $k$, it is instead useful to apply Stokes theorem
and~\eqref{eq:ME1} to replace the integrals over the $A$-cycles with
the $p$ conditions:
\begin{equation}\label{eqn3.22.002}
  \int\limits_{S_j}{\bEta}^-=-a_j.
\end{equation}

In~\cite{EpGr} it is shown that the operator in the center
of~\eqref{eqn3.17.001} is of the form $\Id+K,$ where $K$ is a system of
classical pseudodifferential operators of order $-1.$ The vector space
$\cH^1(\Gamma)\simeq \bbC^{2p}$ and so for any $s> \frac 12,$ the system of
equations defined by~\eqref{eqn3.17.001}--\eqref{eqn3.18.001} is of the form
$\Id+K,$ with $K$ compact, as a map from $W^{s,2}(\Gamma)\times
W^{s,2}(\Gamma)\times\bbC^{2p}$ to itself, and hence Fredholm of index
zero. Let $\cM_{\Gamma,0}$ denote pairs of functions of mean zero on $\Gamma$
and $\cM_{\Gamma,1}$ the $L^2$-closure of
\begin{equation}
  \{(G_0r,q):\: (r,q)\in\cM_{\Gamma,0}\},
\end{equation}
and set $\cM^s_{\Gamma,i}=\cM_{\Gamma,i}\cap W^{s,2}(\Gamma)\times W^{s,2}(\Gamma).$
In~\cite{EpGr} it is also shown that the operator in \eqref{eqn3.17.001} maps 
$\cM^s_{\Gamma,0}$ to $\cM^s_{\Gamma,1},$  for any $s\geq 0.$ From this it
follows easily that, provided $s>1/2,$  the combined system is again Fredholm of index zero as a
map from $\cM^s_{\Gamma,0}\times\bbC^{2p}$ to
$\cM^s_{\Gamma,1}\times\bbC^{2p}.$

The functionals in~\eqref{eqn3.18.001} can be replaced by functionals that
extend as bounded functionals on $L^2$ and are unchanged on closed 1-forms. The
simplest way to do this is to use a basis for $H_1(\Gamma;\bbZ)$ comprised of
smooth embedded, simple closed curves, $\{\gamma_i:\: i=1,\dots,2p\}.$ Each of
the cycles is then homologous to a sum of the form
\begin{equation}
  A_j\sim \sum_{i=1}^{2p}n_{A_j,i}\gamma_i \quad \text{ and }
\quad B_j\sim \sum_{i=1}^{2p}n_{B_j,i}\gamma_i .
\end{equation}
For $j = 1, \ldots, p$, the equations in~\eqref{eqn3.18.001} can then
be replaced with
\begin{equation}\label{eqn3.18.002}
\begin{split}
 \frac{1}{ik}\sum_{i=1}^{2p}n_{A_j,i}
 \int\limits_{\gamma_i}\star_2\cT^{-}_{\bxi}(k)&=
 -a_j\\ \sum_{i=1}^{2p}n_{B_j,i}\int\limits_{\gamma_i}\star_2\cT^{-}_{\bxi}(k)&=-b_j.
\end{split}
\end{equation}

Since $\gamma_i$ is a smoothly embedded simple closed curve in
$\Gamma,$ the tubular neighborhood theorem implies that there is a
smooth family of simple, closed curves $\{\gamma_i^\sigma:\: \sigma\in
[-1,1]\}\subset \Gamma,$ so that:
\begin{enumerate}
\item $\gamma_i=\gamma_i^0.$
\item $\gamma_i^\sigma\cap\gamma_i^{\sigma'}=\emptyset,$ if
  $\sigma\neq \sigma'.$
\item $\cup_{\sigma}\gamma_i^\sigma=U_i$ an open subset of $\Gamma.$
\end{enumerate}
We can replace the integrals over $\gamma_i$ appearing
in~\eqref{eqn3.18.002} with
\begin{equation}
  \frac{1}{2}\int\limits_{-1}^{1}
  \int\limits_{\gamma^{\sigma}_i}\star_2\cT^{-}_{\bxi}(k) \, d\sigma.
\end{equation}
Since the curves $\{\gamma^\sigma_i\}$ are all homologous it is clear
that if $\balpha$ is a closed 1-form then
\begin{equation}
  \frac{1}{2}\int\limits_{-1}^{1}
  \int\limits_{\gamma^{\sigma}_i}\balpha \, d\sigma=
  \int\limits_{\gamma_i}\balpha.
\end{equation}
From the properties of the family of curves it follows that there is a bounded
1-form $\omega_i$ supported in $U_i,$ so that
\begin{equation}
  \frac{1}{2}\int\limits_{-1}^{1}
  \int\limits_{\gamma^{\sigma}_i}\balpha \, d\sigma=
  \int\limits_{\Gamma}\balpha\wedge\star_2\omega_i
\end{equation}
and this therefore is an $L^2$-bounded functional. Replacing the conditions
in~\eqref{eqn3.18.001} with the conditions
\begin{equation}\label{eqn3.18.003}
\begin{split}
 \frac{1}{ik}\sum_{i=1}^{2p}n_{A_j,i}
 \int\limits_{\Gamma}\cT^{-}_{\bxi}(k)\wedge\omega_i&=
 -a_j\\ \sum_{i=1}^{2p}n_{B_j,i}\int\limits_{\Gamma}\cT^{-}_{\bxi}(k)\wedge\omega_i&=-b_j,
\end{split}
\end{equation}
we can extend the analysis of this system to data in $W^{s,2}(\Gamma)$ for
$0\leq s\leq 1/2.$ Thus for $0\leq s,$ the equations \eqref{eqn3.17.001},
\eqref{eqn3.18.003} define a Fredholm map of index zero from
$\cM^s_{\Gamma,0}\times\bbC^{2p}$ to $\cM^s_{\Gamma,1}\times\bbC^{2p}.$

As in our earlier work, if the non-trivial data $(r,q,\bj_H)$ belongs to the
null-space of this system, then it follows that $\bxi^-_t$ is a topologically
trivially harmonic form and therefore $0.$ On the other hand,
Lemma~\ref{lem3.2.002} implies that the solution $(\bxi^-,\bEta^-)$ is
non-trivial. Hence $\bxi^-_t=0$ only arises if $\{k_l\}$ is in the spectrum of
the Maxwell equations in $D$ with the Dirichlet conditions on $\bxi^-$,
i.e. $\bxi^-_t=0$.  The surjectivity for $k\notin\{k_l\}$ follows from the
Fredholm alternative.

Let $\{k_l\}\subset\bbR\setminus\{0\}$ denote the non-zero
exceptional wave numbers for which there exist non-trivial solutions
to
\begin{equation}\label{eqn20.005}
  d\bxi^-=ik_l\bEta^-, \qquad d^*\bEta^-=-ik_l\bxi^-,
\end{equation}
with $\bxi^-_t=0$. Let $d_l$ denote the dimension of the eigenspace
corresponding to $k_l,$ and $n_l$ the dimension of the null-space of
the hybrid system,~\eqref{eqn3.17.001}--~\eqref{eqn3.18.001}.  Clearly
the injectivity of the Debye representation shows that $n_l\leq
d_l$. Suppose that $\{(\bxi_m,\bEta_m):\:1\leq m\leq d_l\}$ is a basis
for this eigenspace. If
\begin{equation}
-id\balpha=k_l\bBeta \quad \text{ and } \quad  id^*\bBeta=k_l\balpha,
\end{equation}
then integrating by parts shows that
\begin{equation}\label{eqn3.13.008}
  \int\limits_{\partial
    D}\balpha\wedge\star_3\overline{\bEta_m}=0\quad \text{ for }1\leq m\leq
  d_l.
\end{equation}
For details, see~\eqref{eqn3.32.002} later on. Since the tangent components
$\bxi_{mt}=0$, Theorem 4.1 in~\cite{ColtonKress} again implies that the 1-forms
$\{[\star_3\bEta_m]_t:\:1\leq m\leq d_l\}$ are linearly independent.
Thus~\eqref{eqn3.13.008} constitutes $d_l$ independent conditions that are
necessary and sufficient for equation~\eqref{eqn20.005}, with the boundary
condition $\bxi^-_t=\balpha_t$, to be solvable. On the other hand, the hybrid
system is solvable if and only if $\balpha_t$ satisfies $n_l$ linear
conditions. This shows that $n_l=d_l$, and that the Debye source representation
is surjective in this case as well.

Suppose that for $k=0$ we have non-trivial data $(r,q,\bj_H)$ in the
null-space of the operator in~\eqref{eqn3.17.001} for which the
integrals over the $B$-cycles in~\eqref{eqn3.18.001}, along with those
in~\eqref{eqn3.22.002}, vanish. In this case $\bxi^-_t$ would be a
topologically trivial harmonic 1-form on $\partial D$, which must
therefore vanish. Hence $\bxi^-=0$ as well. The harmonic 2-form
$\bEta^-$ would vanish on the boundary and represent the trivial class
in $H^2(D,\partial D)$, which implies that it too must vanish. Thus
the field defined by $(r,q,\bj_H)$ would vanish in $D$ and this
contradicts Lemma~\ref{lem3.2.002}.
\end{proof}

These results easily extend to bounded, connected domains $D$
whose boundary $\Gamma = \partial D$ has more than one
component. There is a unique component $\Gamma_0 \subset \Gamma$ that
is also the boundary of the unbounded component of $D^c$. We let
$\{\Gamma_1,\dots,\Gamma_d\}$ denote the other components of $\Gamma$,
each of which is the boundary of a connected component of $D^c$. We
let $E^-_n$ denote the real \emph{exceptional} frequencies for which
there are $k$-Neumann fields for which the restriction $\bxi^-_t$ of
$\bxi^-$ to each component of $\Gamma$ is topologically trivial. By
$k$-Neumann fields, we mean fields that satisfy the \THME[$k$] with
vanishing normal components on the boundary, see~\cite{EpGr}.

To represent solutions to \THME[$k$] in $D$ we use Debye sources
\begin{equation}
  \cS(D)=\{(r_l,q_l,\bj_{Hl}):\: l=0,\dots,d\}\subset
  \prod\limits_{l=0}^{d}\cC^0(\Gamma_l)\oplus
  \cC^0(\Gamma_l)\oplus\cH^1(\Gamma_l)
\end{equation}
on the boundary components. The scalar sources are assumed to have mean
zero on each component of $\Gamma$. As usual, we use these sources
to define 1-forms $\{\bj_l,\bm_l\}$ via equation~\eqref{eqn53}, enforcing
the relations
\begin{equation}\label{eqn3.21.001}
\begin{split}
  \bm_0&=-\star_2\bj_0,\\
  \bm_l&=\star_2\bj_l,
\end{split}
\end{equation}
for $1\leq l\leq d$. Hence, for $l=1,\dots,d$,
\begin{equation}
  \begin{split}\label{eqn3.15.002}
    \bj_0&=ik(d_{\Gamma_0}R_{00}r_0-\star_2d_{\Gamma_0}R_{00}q_0)+\bj_{H0},\\
\bj_l&=ik(d_{\Gamma_l}R_{l0}r_l+\star_2d_{\Gamma_l}R_{l0}q_l)+\bj_{Hl},
  \end{split}
\end{equation}
Using such data it is easy to extend the uniqueness result from
Lemma~\ref{lem3.2.002} to cover the case where $\Gamma$ has
multiple components; surjectivity follows directly as well.
\begin{theorem}\label{prop3.5.01}
If $D$ is a bounded connected region in $\bbR^3$ with smooth boundary
$\Gamma=\Gamma_0\cup\Gamma_1\cup\cdots\cup\Gamma_d$, then, for $k\in
\mathbb C^+$, all solutions to \THME[$k$] in $D$ are uniquely
represented by data in $\cS(D)$.
 \end{theorem}
\begin{remark} Briefly, for any $k\in
\mathbb C^+$, the Debye representation of $\cM_k(D)$ using the data in
$\cS(D)$ and the relations~\eqref{eqn3.21.001} is one-to-one and onto.
\end{remark}
 \begin{proof} The uniqueness result is all that is really needed as we can
   give Fredholm equations of second kind on $\Gamma$ for both $\bxi_t$ and
   $[\star_3\bEta]_t$ in terms of data in $\cS(D)$, and use the argument from
   the proof of Lemma~\ref{lem3.2.002}. This argument is essentially
   identical to that given to prove Theorem 3.4 in~\cite{EpGr2}. We suppose
   that there is data in $\{(r_l,q_l,\bj_{Hl}):\: l=0,\dots,d\}\in\cS(D)$ so
   that the solution, $(\bxi,\bEta)$, specified in $\Gamma^c$ by this data
   vanishes in $D$. We let
  \begin{equation}
    D^c=D_0 \cup D_1\cup\cdots\cup D_d,
  \end{equation}
with $D_0$ the unbounded component and $\Gamma_l=\partial D_l$ for
$l=1,\dots,d$.

We let $(\bxi^0,\bEta^0)$ denote the solution defined by this data in $D_0$ and
$(\bxi^l,\bEta^l)$ the solution in $D_l$. The tangential boundary data for
these solutions in the components of $D^c$ are determined by jump conditions,
and therefore:
\begin{equation}
(\bxi^0_t,i_{\bn}\bEta^0)=(\bj_0,-\bj_0)\quad \text{ and }\quad
  (\bxi^l_t,i_{\bn}\bEta^l)=(\bj_l,\bj_l),
\end{equation}
for $l=1,\dots,d$.  Using the proof of Lemma~\ref{lem3.2.002} we
deduce that $\bj_0=0$, and therefore $r_0=q_0=\bj_{H0}=0$. Using the
argument used to prove Theorem~7.1 in~\cite{EpGr}, we also conclude
that $\bj_l=0$ for $l=1,\dots,d$. This completes the proof of the
uniqueness statement. Surjectivity is proved using the same argument
as above to establish this property when $\partial D$ is connected.
The $k=0$ case follows by combining the argument at the end of the
proof of Lemma~\ref{lem3.2.002} with the proof of Theorem 5.2
in~\cite{EpGr2}.
\end{proof}

\section{The Complex Structure}\label{sec4}
Let $\Omega$ denote a domain in $\bbR^3$ (bounded or unbounded, for
now) with boundary $\Gamma$, and let $k(x)$ be a function defined
in $\Omega$ taking values in $\bbC^+\setminus \{0\}$. For most
applications we take $k(x)$ to be locally constant in the connected
components of $\Omega$, but the first results in this section do not
require any such assumption.

Let $\alpha$ and $\beta$ be non-zero complex constants and
$\cM_{k}(\Omega)$ denote the set of all solutions, defined in
$\Omega$, to the system of equations
\begin{equation}\label{eqn1.001}
  d\bxi = i\alpha k(x)\bEta,\quad d^*\bEta=-i\beta k(x)\bxi.
\end{equation}
We define a map on $\cM_{k}(\Omega)$ by setting
\begin{equation}
  \cJ(\bxi,\bEta)=\left(-\sqrt{\frac{\alpha}{\beta}}\star_3\bEta,\sqrt{\frac{\beta}{\alpha}}\star_3\bxi\right).
\end{equation}
It is an easy calculation to check that if $(\bxi,\bEta)$
solves~\eqref{eqn1.001}, then $\cJ(\bxi,\bEta)$ does as well. Since
the Hodge star-operator satisfies $\star_3^2=\Id$, a further
calculation shows that
\begin{equation}
  \cJ^2(\bxi,\bEta)=-(\bxi,\bEta).
\end{equation}
We summarize this as follows:
\begin{proposition} The map $\cJ$ defines a complex structure on the complex
  vector space $\cM_{k}(\Omega)$.
\end{proposition}
\begin{remark} This observation appears, in passing,  in Section 4.4
  of~\cite{ColtonKress}.
\end{remark}

It is clear from its definition that the map
$\cJ:\cM_{k}(\Omega)\to\cM_{k}(\Omega)$ is complex linear. In
light of that, it defines two projection operators
\begin{equation}
\begin{split}
\cP^{1,0}(\bxi,\bEta)
&=\frac{1}{2}\left[(\bxi,\bEta)-i\cJ(\bxi,\bEta)\right], \\ 
\cP^{0,1}(\bxi,\bEta)
&=\frac{1}{2}\left[(\bxi,\bEta)+i\cJ(\bxi,\bEta)\right].
\end{split}
\end{equation}
We denote the image of $\cM_{k}(\Omega)$ under these projections by
$\cM_{k}^{1,0}(\Omega)$ and $\cM_{k}^{0,1}(\Omega)$
respectively. Another easy calculation demonstrates the following
proposition.
\begin{proposition} The subspace $\cM_{k}^{1,0}(\Omega)$ of
  $\cM_{k}(\Omega)$ is $+i$ eigenspace of $\cJ$ and
  $\cM_{k}^{0,1}(\Omega)$ is the $-i$ eigenspace of $\cJ$.
\end{proposition}

If $(\bxi,\bEta)\in\cM_{k}^{1,0}(\Omega)$, then
$\bEta=-i\sqrt{\frac{\beta}{\alpha}}\star_3\bxi$, and therefore $\bxi$
satisfies the Beltrami equation:
\begin{equation}\label{eqn2.5.01}
  d\bxi=\sqrt{\alpha\beta}k(x)\star_3\bxi.
\end{equation}
If $(\bxi,\bEta)\in\cM_{k}^{0,1}(\Omega)$, then we have that
\begin{equation}
  d\bxi=-\sqrt{\alpha\beta}k(x)\star_3\bxi,
\end{equation}
hence the space $\cM_{k}(\Omega)$ can be decomposed as a direct sum
of solutions to these two Beltrami equations.

\subsection{The Tangent Map}
The space $\cM_{k}(\Omega)$ is an infinite dimensional vector space
whose dependence on the function $k$ is complicated. It is often the
case that an element of this vector space is uniquely determined by
either $\bxi_t=\bxi\restrictedto_{T\partial\Omega}$ or
$[\star_3\bEta]_t=\star_3\bEta\restrictedto_{T\partial\Omega}$. Moreover,
this data can be freely specified to be any tangent 1-form on
$\partial\Omega$. In this case, the operator $\cJ$ defines a complex
structure on the fixed vector space of $1$-forms on $\Gamma$,
$\cC^0(\Gamma;\Lambda^1)$, which depends on the coefficient function
$k(x)$.  On the tangential boundary data, this operator, $\cJ^b$, is
given by
\begin{equation}
  \cJ^b(\bxi_t)=-\sqrt{\frac{\alpha}{\beta}}[\star_3\bEta]_t.
\end{equation}

This operator is the analogue for the time-harmonic Maxwell equations of the
Dirichlet-to-Neumann map for a scalar elliptic equation. The remaining Cauchy
data consists of the normal components of $\bxi$ and $\bEta$, which in this
formulation are just restrictions of the 2-forms to the boundary:
\begin{equation}
  \star_3\bxi\restrictedto_{\partial\Omega} \quad \text{ and } \quad
  \bEta\restrictedto_{\partial\Omega}.
 \end{equation}
Using equation~\eqref{eqn1.001} these quantities can be determined
from $\bxi_t$ and $[\star_3\bEta]_t$:
  \begin{equation}\label{eqn2.9.03}
     \star_3\bxi\restrictedto_{\partial\Omega}=-\frac{d_{\partial\Omega}[\star_3\bEta]_t}{i\beta
       k(x)} \quad \text{ and } \quad \bEta\restrictedto_{\partial\Omega}=
     \frac{d_{\partial\Omega}\bxi_t}{i\alpha k(x)}.
  \end{equation}
From these relations it is clear why $\cJ^b$ should be understood as the
analogue of the Dirichlet-to-Neumann map.

As the dependence of $\cJ^b$ on $k$ is now quite important we denote
this operator on $\cC^0(\partial\Omega;\Lambda^1)$ by $\cJ^b_{k}$. That
$[\cJ_{k}^b]^2=-\Id$ is far from obvious, though nonetheless true. The
dependence of $\cJ_{k}^b$ on $k$ is quite complicated. As $k$ tends to
zero the family of operators diverges somewhat like the matrices
\begin{equation}
  \left(\begin{matrix} 0& -\frac{1}{k}\\ k & 0\end{matrix}\right)\quad k\in\bbC\setminus\{0\}.
\end{equation}
In another paper we will study the behavior of this family of operators as the
constant function $k$ tends to zero and infinity in the closed upper half
plane. For now we observe that if $\bxi_t$ uniquely determines the solution
to~\eqref{eqn1.001}, then the condition
\begin{equation}
  \cJ_k^b(\bxi_t)=\pm i\bxi_t
\end{equation}
holds if and only if
\begin{equation}
  \bEta=\mp i\sqrt{\frac{\beta}{\alpha}}\star_3\bxi.
\end{equation}
The former condition is an immediate consequence of the latter. On the other
hand if the boundary condition holds, then $(\bxi,\bEta)$ and $\mp
i\cJ(\bxi,\bEta)$ are solutions to~\eqref{eqn1.001} with the same tangential
boundary data. By uniqueness they must agree throughout $\Omega$.

\subsection{$k$-Neumann Fields and Beltrami Fields}\label{ss2.2}

Let us first examine the case in which $\Omega$ is an unbounded
region.  If $\Omega$ is the complement of a bounded region $D$, and
$k(x)$ is constant outside of a compact set, $K$, then $d^*\bxi=0$ in
$K^c\cap\Omega$ and we can use~\eqref{eq:radcond} to define the subset
of outgoing solutions to~\eqref{eqn1.001}. We denote this set by
$\cM^{\Out}_{k}(\Omega)$. From the definition of $\cJ$ it is easy
to see that $\cJ$ maps $\cM^{\Out}_{k}(\Omega)$ to itself.

In applications to fluid mechanics and plasma physics
\cite{bauer,hudson,hudson2,marsh} a particular subspace of solutions
to~\eqref{eqn2.5.01} plays an important role: those solutions with
vanishing normal components. If we let $\bn$ denote the unit outward
normal vector field along $\Gamma$, then in the exterior form
representation this condition is expressed by
\begin{equation}\label{eqn2.7.01}
  i_{\bn}\bxi=0 \quad \text{ and } \quad i_{\bn}\star_3\bEta=0.
\end{equation}
As their normal components vanish, these fields exert no outward force
on the boundary, and are therefore called \emph{force-free} fields.
Once again it follows easily from the definition that $\cJ$ preserves
this condition. In this case, equation~\eqref{eq:ME1} implies that the
tangent components of these fields $(\bxi_t,[\star_3\bEta]_t)$ are
closed forms on the boundary $\partial\Omega$. We say that a
$k$-Neumann field is \emph{topologically trivial} if $\bxi_t$ defines
the trivial class in $H^1_{\dR}(\Omega)$. One could also say that a
$k$-Neumann field is topologically trivial if $[\star_3\bEta]_t$
defines the trivial class in $H^1_{\dR}(\partial\Omega)$, but for the
moment we just consider the $\bxi$-component. In
Section~\ref{subsec7.4} we consider these conditions in a broader
context. In the case $\Omega$ is an unbounded region, and $k$ is a
constant, then it is proved in~\cite{EpGr} that an outgoing
$k$-Neumann field for which either component is topologically trivial
is automatically zero.

If $D$ is a bounded domain, or if we restrict attention to outgoing
solutions in $\Omega$, then it is known that the space of solutions
to~\eqref{eqn1.001} that satisfy the boundary
condition~\eqref{eqn2.7.01} is finite dimensional. In the case that
$k(x)$ is piecewise constant these solutions are called $k$-Neumann
fields \cite{EpGr}, which we denote by $\cN_k(D)$, $\cN_k(\Omega)$,
respectively. Since $\cJ$ preserves these subspaces, it again defines a
complex structure and so these vector spaces split into the eigenspaces
of $\cJ_k:$
\begin{equation}\label{eqn2.14.001}
\begin{split}
  \cN_k(\Omega)&=\cN_k^{1,0}(\Omega)\oplus \cN_k^{0,1}(\Omega),\\
  \cN_k(D)&=\cN_k^{1,0}(D)\oplus \cN_k^{0,1}(D).
\end{split}
\end{equation}
We can represent these vectors spaces as $(\bxi,\mp i\star_3\bxi)$,
from which it is apparent that $\cN_k^{1,0}(\Omega) \simeq
\cN_{-k}^{0,1}(\Omega)$. For real $k$, the equations $d\bxi=\pm
k\star_3\bxi$ are real, so we can take bases of the forms $\{(\bxi,\mp
i\star_3\bxi)\}$, with $\bxi$ real.

Now suppose that $\Omega$ is the complement of a bounded region $D$
with $\cC^1$ boundary $\Gamma$, and $k \in \bbC^+$ is constant.  The
space of $k$-Neumann fields is well understood to be closely related
to the topology of $\partial\Omega$. If the total genus of the
components of $\partial\Omega$ is $p$, then the remarks above imply
that
\begin{equation}\label{eqn2.15.001}
  \dim\cN_k(\Omega)=2p.
\end{equation}
This will allow us to prove the following result:
\begin{proposition} 
If $\Omega$ is the complement of a finite union of smoothly bounded
regions, then
\begin{equation} 
  \dim \cN_k^{1,0}(\Omega)=\dim \cN_k^{0,1}(\Omega)=p
\end{equation}
for all $k\in \mathbb C^+\setminus\{0\}$. 
\end{proposition}
\begin{proof}
  As shown in~\cite{EpGr}, an outgoing solution to \THME[k], with $k\neq 0,$ is
  uniquely determined by the data $(\bxi_n,[\star\bEta]_n,\ba,\bb)$, where
  $\ba, \bb\in\bbC^p$ is the topological data, defined for $j=1,\dots,p$ by
\begin{equation}\label{eqn4.17.002}
a_j=\int\limits_{A_j}\bxi_t \quad \text{ and } \quad 
b_j=\int\limits_{B_j}\bxi_t.
\end{equation}
Elements of $\cN_k(\Omega)$ correspond to data in the subspace
$W=\{(0,0,\ba,\bb)\}$. Let
\begin{equation}
N_{A,k}:(r,q,\bj_H)\to (\bxi_n,[\star\bEta]_n,\ba,\bb)
\end{equation}
be the map defined by the Debye source representation. As a function of $k,$
this is an analytic family of injective Fredholm operators.

In Section~\ref{sec6.0} it is shown that  elements of $\cM^{1,0}_k(\Omega)$ are specified by
Debye source data of the form $V^{1,0}=\{(r,ir,\bj_H^{1,0})\}$, where
$r$ has mean zero on every component of $\Gamma$ and
$i\bj_H^{1,0}=-\star_2\bj_H^{1,0}$. Similarly, elements of
$\cM^{0,1}_k(\Omega)$ are specified by Debye source data of the form
$V^{0,1}=\{(r,-ir,\bj_H^{0,1})\}$, with
$i\bj_H^{0,1}=\star_2\bj_H^{0,1}$. A moments consideration shows that
\begin{equation}
\dim\cN_k^{1,0}(\Omega)=
\dim\Ker(\Id-\pi_W)\circ N_{A,k}\restrictedto_{V^{1,0}},
\end{equation}
where $\pi_W$ denotes the orthogonal projection onto $W$.  A similar
statement exists for $\dim\cN_k^{0,1}(\Omega)$. For generic $k\in
\mathbb C^+$ we know that $\dim\cN_k^{1,0}(\Omega)=p$. Moreover,
standard results about analytic families of operators show that for
$k_0\in \mathbb C^+$ we have
\begin{equation}
\begin{split}
  \dim\Ker(\Id-\pi_W)\circ N_{A,k_0}\restrictedto_{V^{1,0}}&\geq\limsup_{k\to
    k_0}
\dim\Ker(\Id-\pi_W)\circ N_{A,k_0}\restrictedto_{V^{1,0}} \\
&=p.
\end{split}
\end{equation}
The same statement holds for $\dim\Ker(\Id-\pi_W)\circ
N_{A,k_0}\restrictedto_{V^{0,1}}$. These semi-continuity results,
along with~\eqref{eqn2.14.001} and~\eqref{eqn2.15.001}, shows that for
all $k\in \mathbb C^+$,
\begin{equation}
\dim \cN_k^{1,0}(\Omega)=\dim \cN_k^{0,1}(\Omega)=p.
\end{equation}

\end{proof}

From this it follows immediately that for any $k\in \mathbb C^+$ the
equation
\begin{equation}\label{eqn2.14.01}
  d\bxi = k\star_3\bxi \quad \text{ with } \quad i_{\bn}\bxi=0
\end{equation}
has a $p$-dimensional space of outgoing solutions. If $k=0$, and we
explicitly append the divergence condition, $d^*\bxi=0$, then this
statement remains true. In this case, the solution space is precisely
$\cH^1(\Omega)$.

If $D$ is now a bounded region with boundary $\partial D$ of genus $p$, then
there is a countable collection of real wave numbers $E^-_n=\{k_l\}$ so that
there exist topologically trivial $k_l$-Neumann fields in $D$. See
Section~\ref{subsec7.4}.  We call these frequencies \emph{$k$-Neumannn resonances}.
To see that such resonances must be real, we observe that for any solution
to~\eqref{eq:ME1} in $D$ we have the integration by parts formula
\begin{equation}
 \int\limits_{D}d\bxi^-\wedge\star_3\overline{d\bxi^-}
-  k^2\int\limits_{D}\bxi^-\wedge\star_3\overline{\bxi^-}
=
-ik\int\limits_{\partial D}\star_3\bEta^-\wedge\overline{\bxi^-}.
\end{equation}
If $(\bxi^-,\bEta^-)$ is a $k$-Neumann field and either $\bxi^-_{t}$
or $[\star_3\bEta^-]_t$ is topologically trivial, then the right hand
side is zero. This shows that $k^2$ is a positive real number. 

For $k\in \mathbb C^+\setminus E^-_n$, we have $\dim \cN_k(D)=2p$, as
in the exterior case. Since $E^-_n$ is a discrete set, we can again
show that generically
\begin{equation}
  \dim\cN_k^{1,0}(D)= \dim\cN_k^{0,1}(D)=p.
\end{equation}
It follows exactly as in the unbounded case that this holds for all $k
\notin E^-_n$. Once again for each $k\in \mathbb C^+\setminus E^-_n$
the equations in~\eqref{eqn2.14.01} have a $p$ dimensional solution
space. In Section~\ref{sec6.0} we give a more precise description of
the exceptional set. Kress \cite{kress2} proved a similar result for
the case of a torus.

Thus we see that the problem of finding Beltrami fields can be reduced
to that of finding $k$-Neumann fields, and vice versa. In the next
section we show that for frequencies in $E^-_n$, $\dim\cN_k(D)\geq
2p$. Indeed for the case that $D$ is a round ball, we show that every
Dirichlet eigenvalue for the Laplacian on scalars is also $k$-Neumann
resonance, and the dimensions of these eigenspaces grow without bound
as the eigenvalues tends to infinity.  In the simply connected case,
the topological bound does not pertain at $k$-Neumann
resonances. Corollary~\ref{cor7.10.001} gives a more precise
description of the exceptional set, from which it seems plausible that
it is empty if the genus of $\partial D$ is greater than 1. On the
other hand, we also show that for a torus of revolution the
exceptional set is infinite.

\begin{remark} While the complex structure described here is not explicitly
  defined in~\cite{PicardCurlED}, a connection between solutions to
  the \THME[$k$] satisfying $\star_3\bEta=i\bxi$ and eigenfunctions of
  $\star_3 d$ is utilized. Picard's paper also considers self-adjoint
  extensions of the curl operator in exterior domains.
\end{remark}

\section{Debye Sources and the Complex Structure}

We next state some results concerning the uniformization of the
complex structure. Let $D$ denote a bounded, connected domain in
$\bbR^3$ with smooth boundary $\partial D$.  From~\eqref{tngtds0.01},
the relationship between the Debye representation and the complex
structure on solutions to \THME[$k$] becomes quite evident:
\begin{theorem}\label{thm3.1.001} 
  Let $(\bj,\bm)$ be 1-forms on $\partial D$. If, for $k\neq 0$,
  $(\bxi^{\pm},\bEta^{\pm})$ are the solutions to \THME[$k$] defined by
  $(\bj,\bm)$, i.e., with $(r,q)$ defined by~\eqref{eqn53}, then the fields
  defined by $\cJ^D(\bj,\bm)= (-\bm,\bj)$ are $\cJ_k(\bxi^{\pm},\bEta^{\pm})$.
\end{theorem}
\begin{proof} 
  This is a simple algebraic consequence of the Debye source
  representation formula~\eqref{eqn29} and the definitions of the
  potentials~\eqref{srfint2}.
\end{proof}
Briefly, using the Debye source representation, the complex structure on
solutions to \THME[$k$] is represented on a \emph{fixed} vector space by an
operator that is independent of the frequency! This is therefore a
uniformization theorem for the complex vector spaces $\cM_k(\Omega)$. The
theorem is correct even for frequencies where the map to tangent electric (or
magnetic) component is not injective. In this generality, the converse statement is
not correct, as there is non-zero Debye source data $(r,q,\bj,\bm)$
satisfying~\eqref{eqn53} that defines the zero-field $(0,0)$ in $D$. To obtain
the converse statement we need to posit relations between $\bj$ and $\bm$ like
those in~\eqref{eqn3.21.001}. As a corollary of Theorem~\ref{thm3.1.001} and
Theorem~\ref{prop3.5.01} we have:
\begin{corollary}\label{cor5.2.002}
  Let $D$ be a bounded connected region in $\bbR^3$ satisfying the
  hypotheses of Theorem~\ref{prop3.5.01}. If $k\in \mathbb C^+$ and
  $(\bxi^-,\bEta^-)$ is a time-harmonic Maxwell field in $D$ defined
  by Debye source data
  \begin{equation*}
  [(r_0,q_0,\bj_{H0}),(r_1,q_1,\bj_{H1}),\dots,(r_d,q_d,\bj_{Hd})],
  \end{equation*}
  with $\bj$ defined by~\eqref{eqn3.15.002} and $\bm$
  by~\eqref{eqn3.21.001}, then $\cJ_k(\bxi^-,\bEta^-)$ is uniquely
  defined by the Debye source data
  \begin{equation*}
  [(q_0,-r_0,\star_2\bj_{H0}),
    (q_1,-r_1,-\star_2\bj_{H1}),\dots,(q_d,-r_d,-\star_2\bj_{Hd})].
  \end{equation*}
\end{corollary}

\begin{remark} 
We let $\cJ^{\cS(D)}(r,q,\bj_H)$ denote Debye source representation
for the complex structure, which is again frequency independent:
\begin{multline}
 \cJ^{\cS(D)}:\: [(r_0,q_0,\bj_{H0}),(r_1,q_1,\bj_{H1}),\dots,(r_d,q_d,\bj_{Hd})]\mapsto\\
[(q_0,-r_0,\star_2\bj_{H0}),(q_1,-r_1,-\star_2\bj_{H1}),\dots,(q_d,-r_d,-\star_2\bj_{Hd})].
\end{multline}
\end{remark}

This corollary  leads directly to a result giving the dimension of the
space of $k$-Neumann fields on $D$ when $\partial D$ is not
connected.
\begin{corollary} 
Let $p_l$, $l=0,\dots,d$, denote the genera of the boundary components
of $D$, and $p=p_0+\cdots+p_d$. For $k\in \mathbb C^+\setminus E^-_n$,
the space of $k$-Neumann fields defined in $D$, $\cN_k(D)$, is $2p$
dimensional.  For $k\in E^-_n$ this dimension is at least $2p$.
\end{corollary}
\begin{proof}
  A small modification of the argument used to prove Theorem 4.1
  in~\cite{EpGr2} along with Theorem~\ref{prop3.5.01} show that for all
  but the countably many real frequencies in $E^-_n$, the normal
  equations~\eqref{nrmds0} augmented with the algebraic conditions
  in~\eqref{eqn4.17.002} define a one-to-one and onto map,
  \begin{equation}
    \{(r_l,q_l,\bj_H)\} \mapsto [\bxi^-_n,\bEta^-_n,(\ba,\bb)],
  \end{equation}
  where $\ba,\bb\in\bbC^p$. The representation in terms of the data in
  $\cS(D)$ is injective. This shows that for $k\notin E^-_n$, the
  dimension of the space of solutions to \THME[$k$] in $D$ with given
  normal data, $(\bxi^-_n,\bEta^-_n)$, equals $2p$:
  \begin{equation}
    \sum_{l=0}^d\dim\cH^1(\Gamma_l)=2p.
  \end{equation}
The space of $k$-Neumann fields is just the special case of data of
the form $\{[0,0,(\ba,\bb)]:\: \ba,\bb\in \bbC^p\}$.

For frequencies $k\in E^-_n$, where the augmented system of normal equations
has a non-trivial null-space, there is at least a $2p$ dimensional space of
$k$-Neumann fields. If the dimension of the null-space is $n$, then this
already constitutes an $n$-dimensional subspace of $\cN_k(D)$. By the Fredholm
alternative, the augmented normal system is solvable for data satisfying $n$
linear conditions, which means that we get an additional subspace of $\cN_k(D)$
with dimension at least $2p-n$.
\end{proof}

\begin{remark}
   Above we explain why the exceptional frequencies are real. This
   suggests that there should be a self-adjoint boundary value problem
   whose spectrum contains these frequencies.  There is, in fact, a
   family of such self-adjoint problems among whose spectra the
   frequencies in $E^-_n$ occur.  This indicates that the problem of
   characterizing the frequencies for which there is a topologically
   trivial solution, and determining the dimensions of the space of
   $k$-Neumann fields for these frequencies, may well be quite
   difficult. These boundary conditions are explained in
   Section~\ref{sec6}.
\end{remark}

\section{Finding Beltrami Fields in Bounded Domains}\label{sec6.0}

Theorem~\ref{thm3.1.001} shows that the boundary data for the space
$\cM_k^{1,0}(D)$ consists of Debye source data for which
$\cJ^D(\bj,\bm)=i(\bj,\bm)$, that is 
\begin{equation}
\bm=-i\bj,
\end{equation} and $\cM_k^{0,1}(D)$
consists of Debye source data for which
$\cJ^D(\bj,\bm)=-i(\bj,\bm)$. In this section we consider the case of
a bounded region $D$ with boundary $\partial
D= \Gamma = \Gamma_0\cup\Gamma_1\cup\cdots\cup\Gamma_d$.

If we represent $\bj$ and $\bm$ in terms of scalar Debye sources and harmonic
1-forms, then we see that if 
\begin{equation}\label{eqn3.28}
  \bm=\pm i\bj,
\end{equation}
 then
\begin{equation}\label{eqn3.28a}
  q=\mp ir,
\end{equation}
so that for $l=1,\dots,d$
\begin{equation}\label{eqn3.29}
\begin{split}
\bj_0&=ik(d_{\Gamma_0}R_{00}r_0\pm
i\star_2d_{\Gamma_0}R_{00}r_0)+\bj_{H0},\\ \bj_l&=ik(d_{\Gamma_l}R_{l0}r_l\mp
i\star_2d_{\Gamma_l}R_{l0}r_l)+\bj_{Hl}.
\end{split}
\end{equation}
We also have relations satisfied by the harmonic components:
\begin{equation}\label{eqn3.30}
  \star_2\bj_{H0}=\mp i\bj_{H0} \quad \text{ and } \quad
\star_2\bj_{Hl}=\pm i\bj_{Hl},
\end{equation}
for $l=1,\dots,d$.
The  conditions in~\eqref{eqn3.30} are the requirements that harmonic 1-forms
$\{\bj_{Hl}\}$ be either holomorphic or anti-holomorphic $(1,0)$-forms, or $(0,1)$-forms.

Corollary~\eqref{cor5.2.002} has the following useful corollary.
\begin{corollary} 
If $D$ is bounded, then all solutions to the \THME[$k$] in
  $D$ belonging to $\cM_k^{1,0}(D)$ are given by data of the form
  $(\bj,-i\bj)$, and solutions in $\cM_k^{0,1}(D)$ are given by data of the
  form $(\bj,i\bj)$, where these vector sources satisfy the relations
  in~\eqref{eqn3.29} and~\eqref{eqn3.30}.
\end{corollary}

We can use this representation to find Beltrami fields. For data
satisfying~\eqref{eqn3.28} the equation for the normal component of the
electric field reduces to:
\begin{equation}\label{eqn5.23.002}
  \cN_{\bxi}(r,\bj_H)=-\frac{r}{2}-K_0r+ikK_{2,n}\bj\mp iK_3\bj=f.
\end{equation}
If the total genus of $\partial D$ is $p$, then the space of harmonic
1-forms satisfying the conditions in~\eqref{eqn3.30} is $p$
dimensional and so this equation must be augmented with $p$ algebraic
conditions. For $k=0$, it is easy to see which additional conditions
lead to an invertible Fredholm system of second kind. The same
additional conditions therefore suffice for $k$ in the complement of a
discrete set. In fact there are many possible choices as we explain in
Sections~\ref{sec7.1}--\ref{sec7.2}.

We let $\{A_j:\:j=1,\dots,p\}\cup \{B_j:\:j=1,\dots,p\}$ denote a
basis for $H_1(\partial D)$, where for each $A_j$, there is a chain
$S_j\subset D$ with $\partial S_j=A_j$. For $k\notin E^-_n$, a
$k$-Neumann field $(\bxi,\bEta)$ is determined by the additional data
\begin{equation}\label{eqn3.33}
      a_j=\frac{1}{ik}\int\limits_{A_j}\bxi,
\end{equation}
\begin{equation}\label{eqn3.33.001}
 b_j=\int\limits_{B_j}\bxi
    \end{equation}
for $j=1,\dots, p$.
Using the equation $d\bxi=ik\bEta$, the integrals over the $A$-cycles
can be replaced with
\begin{equation}\label{eqn5.26.002}
 a_j= \int\limits_{S_j}\bEta.
\end{equation}
Indeed, as shown in~\cite{EpGr2}, these integrals can be specified
arbitrarily as $k\to 0$.

If we now restrict attention to $(\bxi,\bEta)\in \cN^{1,0}_k(D)$, then
$\bEta=-i\star_3\bxi$ and the harmonic 1-forms
satisfy~\eqref{eqn3.30}, hence there are only $p$ additional
parameters. The area integrals in~\eqref{eqn5.26.002} can be
re-expressed as:
\begin{equation}\label{eqn3.35}
  a_j=-i\int\limits_{S_j}\star_3\bxi, 
\end{equation}
for $j=1,\dots, p$. We therefore have the following Fredholm system of
second kind for solutions to the Beltrami equation
$d\bxi=k\star_3\bxi$
\begin{equation}\label{eqn6.10.004}
  \begin{split}
      \cN_{\bxi}(r,\bj_H)&=f\\
\int\limits_{S_j}\star_3\bxi &=a_j,
  \end{split}
\end{equation}
for $j=1,\dots, p.$ 
Here $f$ is any function with mean zero on every component of
$\partial D$ and $(a_1,\dots,a_p)\in\bbC^p$.

As $k\to 0$ the elements of $ \cN^{1,0}_k(D)$ converge to harmonic
representatives of $H^1_{\dR}(D)$, which is a $p$-dimensional vector
space. Poincar\'e duality implies that if $\bxi$ is such a harmonic
1-form, then $\star_3\bxi$ is a harmonic representative of an element
of $H^2_{\dR}(D,\partial D)$. The integrals over the $B$-cycles
in~\eqref{eqn3.33.001} uniquely determine elements of $H^1_{\dR}(D)$,
while the integrals in~\eqref{eqn3.35} uniquely determine elements of
$H^2_{\dR}(D,\partial D)$. Thus Theorem~\ref{prop3.5.01} shows that at
$k=0$ either set of data suffices to uniquely determine an element of
$\cN^{1,0}_0(D)$. The same line of reasoning applies to
$\cN^{0,1}_0(D)$. This analysis therefore shows that
\begin{equation}\label{eq6.9.002}
  \dim\cN^{1,0}_0(D)=\dim\cN^{0,1}_0(D)=p.
\end{equation}

As the equations in~\eqref{eqn6.10.004} depend analytically on $k$,
and are injective at $k=0$, it is clear that for all but a countable
collection of wave numbers, those in $E^{-(1,0)}_n$ and
$E^{-(0,1)}_n$, the integrals in either~\eqref{eqn5.26.002}
or~\eqref{eqn3.33.001} determine elements of $\cN^{1,0}_k(D)$ and
$\cN^{0,1}_k(D)$, respectively.  Since for all $k$
\begin{equation}
  \dim\cN_k(D)=\dim\cN^{1,0}_k(D)+\dim\cN^{0,1}_k(D),
\end{equation}
it must be the case that for sufficiently small $k$
\begin{equation}\label{eqn5.28.002}
  \dim\cN^{1,0}_k(D)=\dim\cN^{0,1}_k(D)=p.
\end{equation}
If $k\in E^{-(1,0)}_n$, then the standard Fredholm alternative
argument applied to~\eqref{eqn6.10.004}, along with
Theorem~\ref{prop3.5.01}, show that
\begin{equation}
  \dim\cN^{1,0}_k(D)\geq p,
\end{equation}
with an analogous result for $k\in E^{-(0,1)}_n$ and
$\dim\cN^{0,1}_k(D)$.  This implies that the equalities
in~\eqref{eqn5.28.002} hold for $k\in \mathbb C^+\setminus
E^-_n$. Hence, it is only for frequencies in $E^-_{n}$ that there may
be a surplus of solutions to the Beltrami equations with vanishing
normal data.

For frequencies $k\in E^{-(1,0)}_n$ there are solutions $\bxi^-$ to the
equation
\begin{equation}
  d\bxi^-=k\star_3\bxi^-
\end{equation}
with $i_{\bn}\bxi^-=0$ and vanishing fluxes
\begin{equation}
  \int\limits_{S_j}\star_3\bxi^-=\frac{1}{k}\int\limits_{A_j}\bxi^-=0.
\end{equation}
In the next section we show that there is always an infinite sequence
of real frequencies for which such solutions exist, and place this
problem in the larger context of a family of self-adjoint boundary
value problems for the operator $\star_3d$ acting on divergence-free
1-forms in a bounded domain. This family of operators is parametrized by Lagrangian
subspaces of $H^1_{\dR}(\partial D)$ with respect to the wedge product
pairing.

\section{Non-local Self-Adjoint Boundary Conditions}\label{sec6}

The operator $\cB$ acting on divergence-free 1-forms ($d^*\bxi=0$),
\begin{equation}
\cB\bxi=\star_3d\bxi,
\end{equation}
and the operator $\cL$ acting on the space of divergence-free pairs
$(\bxi,\bEta)$, i.e. those satisfying $d^*\bxi=0$ and $d\bEta=0$,
\begin{equation}
\cL(\bxi,\bEta)=(id^*\bEta,-id\bxi),
\end{equation}
are formally self-adjoint, or symmetric with respect to the standard
$L^2$-structures:
\begin{equation}
  \|\bxi\|_{L^2}^2=\int\bxi\wedge\star\overline\bxi\text{ and }
\|(\bxi,\bEta)\|_{L^2}^2=\int\bxi\wedge\star\overline{\bxi}+
\int\bEta\wedge\star\overline{\bEta}.
\end{equation}
In this section we consider
certain non-local boundary conditions that define self-adjoint
operators with these formal expressions. This question has been
considered by several authors. A result equivalent to
Theorem~\ref{thm7.2.001} below appears in~\cite{HKTSACurlBD}, with an
earlier related result
in~\cite{PicardCurlBD}. The article~\cite{HKTSACurlBD} has an extensive
bibliography of papers that consider self-adjoint boundary conditions
for the curl operator.

The analysis of the Beltrami operator is greatly facilitated by some standards
estimates that arise in the study of the Laplacian on forms. We use
$W^{s,2}(D),$ for $s\in\bbR,$ to denote the $L^2$-Sobolev spaces of
distributions on $D.$ The results we use are the estimates (9.16) and (9.32)
from Chapter 5 of~\cite{TaylorI}. If $D$ is a smoothly bounded domain in
$\bbR^n,$ then there is a constant $C$ so that $k$-forms satisfy the following
estimate:
\begin{equation}\label{eqn7.8.007}
  \|\bxi\|^{2}_{W^{1,2}(D)}\leq
  C\left[\|d\bxi\|^{2}_{L^2(D)}+\|d^*\bxi\|^{2}_{L^2(D)}+
\|i_{\bn}\bxi\|^{2}_{W^{1/2,2}(D)}+\|\bxi\|^{2}_{L^2(D)}\right]
\end{equation}
This holds for $\bxi\in W^{1,2}(D),$ but, as follows from the proof
of~\eqref{eqn7.8.007} via potential theory, it should be understood that the
finiteness of the right hand side implies that $\bxi\in W^{1,2}(D).$ To prove
higher norm estimates we use the inequality, again for $k$-forms defined in
$W^{j+1,2}(D):$
\begin{equation}\label{eqn7.9.007}
  \|\bxi\|^{2}_{W^{j+1,2}(D)}\leq
  C\left[\|\Delta\bxi\|^{2}_{W^{j-1,2}(D)}+
\|i_{\bn}\bxi\|^{2}_{W^{j+1/2,2}(\pa D)}+
\|i_{\bn}d\bxi\|^{2}_{W^{j-1/2,2}(\pa D)}+\|\bxi\|^{2}_{L^2(D)}\right].
\end{equation}
As before, the finiteness of the right hand side implies that $\bxi\in W^{j+1,2}(D)$

The following classical lemma on the existence of distributional boundary values
is very useful in the sequel.
\begin{lemma}\label{lem7.0.002} 
Let $D$ be a smoothly bounded domain in $\bbR^n$, and $m \leq n$. 
Let $\bxi$ be an
$L^2(D)$ $m$-form defined in $D$.
\begin{enumerate} 
\item If $d\bxi$ is also in $L^2(D)$, then $\bxi_t$ is well defined as
  an element of $W^{-1/2,2}(\partial D)$.
\item If $d^*\bxi$ is also in $L^2(D)$, then $i_{\bn}\bxi$ is a well defined
  element of $W^{-1/2,2}(\partial D)$.
\end{enumerate}
These results also hold if $L^2(D)$ is replaced by $W^{s,2}(D)$, with
$W^{-1/2,2}(\partial D)$ replaced by $W^{s-1/2,2}(\partial D)$.
\end{lemma}
\noindent
The lemma is proved in~\cite{TaylorI}

It is a consequence of the lemma that if $d\bxi\in L^2(D)$, then the equation
\begin{equation}\label{eqn7.1.005}
  d_{\partial D}\bxi_t=0,
\end{equation}
makes distributional sense. The Hodge decomposition of 1-forms defined on $\partial D$,
given by
\begin{equation}\label{eqn7.2.005}
  \theta=d_{\partial D}f+d_{\partial D}^*\chi+\omega_H,\quad \text{
    where }\omega_H\in\cH^1(\partial D),
\end{equation}
extends to 1-forms with distributional coefficients. As the harmonic
forms $\cH^1(\partial D)$ provide smooth representatives for elements
of $H^1_{\dR}(\partial D)$ it makes sense to discuss the cohomology
classes of distributional solutions to~\eqref{eqn7.1.005}. If
$\{\theta_j:\:j=1,\dots,2p\}$ is an orthonormal basis for
$\cH^1(\partial D)$, then $\balpha$, a closed distributional 1-form,
is cohomologous to
\begin{equation}
  \theta=\sum\limits_{j=1}^{2p}\alpha_j\theta_j
\end{equation}
if and only if for $j=1,\dots,2p$,
\begin{equation}
  \alpha_j=\int\limits_{\partial D}\balpha\wedge\star_2\theta_j.
\end{equation}
We let $[\balpha]=[\theta]\in H^1_{\dR}(\partial D)$ denote the cohomology class of
$\balpha$.  For smooth closed 1-forms $\balpha_1,\balpha_2,$ the value of the wedge
product pairing
\begin{equation}
  \int\limits_{\partial D}\balpha_1\wedge\balpha_2
\end{equation}
only depends on the cohomology classes $[\balpha_1],[\balpha_2],$ which
provides an extension of this pairings  to
distributionally closed forms. If $\balpha_1, \balpha_2$ are distributional
solutions to~\eqref{eqn7.1.005}, then
\begin{equation}
  \int\limits_{\partial D}\balpha_1\wedge\balpha_2\overset{d}{=}
\int\limits_{\partial D}\tbalpha_1\wedge\tbalpha_2,
\end{equation}
where $\tbalpha_1, \tbalpha_2$ denote smooth representatives of these
cohomology classes.

\subsection{The Beltrami Operator}\label{sec7.1}
We begin with a study of Beltrami fields, considering
the operator
\begin{equation}
  \cB\bxi=\star_3d\bxi,
\end{equation}
acting on the space of divergence-free 1-forms on a bounded
connected domain $D$.  We let $\cE_1(D)$ be the $L^2(D)$-closure of the
divergence-free 1-forms on $D$, with the inner product defined by:
\begin{equation}
  \langle\bxi_1,\bxi_2\rangle=\int\limits_{D}\bxi_1\wedge\star_3\bxi_2.
\end{equation}
It is clear that, distributionally, $d^*\cB\bxi=0$. A simple integration by
parts shows that
\begin{equation}
  \langle
  \cB\bxi_1,\bxi_2\rangle=\langle\bxi_1,\cB\bxi_2\rangle-\int\limits_{\partial D}\bxi_1\wedge\bxi_2.
\end{equation}
We could also use complex 1-forms and the Hermitian
inner product:
\begin{equation}
  \langle\bxi_1,\bxi_2\rangle=\int\limits_{D}\bxi_1\wedge\star_3\overline{\bxi_2}.
\end{equation}
As the operator $\cB$ is real and the boundary conditions are
self-adjoint, it suffices to consider real 1-forms and therefore we
restrict our attention to the real inner product.

Let $\Lambda_H^1(\partial D)$ denote the collection of Lagrangian
subspaces of $H^1_{\dR}(\partial D)$ with respect to the standard
symplectic pairing on $\partial D$. If the genus of $\partial D$ is
$p$, then an element $\lambda\in\Lambda_H^1(\partial D)$ is a
$p$-dimensional subspace of $H^1_{\dR}(\partial D)$ such that if
$\omega_1,\omega_2$ represent cohomology classes in $\lambda$, then
\begin{equation}
  \int\limits_{\partial D}\omega_1\wedge \omega_2=0.
\end{equation}
If $W$ is a $2p$-dimensional symplectic vector space, then the set of
Lagrangian subspaces of $W$ is identified with a homogeneous space
of dimension $p(p+1)/2$.

If we let $\lambda$ be a Lagrangian subspace of $H^1_{\dR}(\partial D)$, then
the operator $\cB$ acting on divergence-free 1-forms $\bxi$ with $d_{\pa
  D}\bxi_t=0$ and $[\bxi_t]\in \lambda$ is formally symmetric.  Each of these
operators has an infinite dimensional null-space: the set of  1-forms $\{du\}$, with $u\in
W^{1,2}(D)$ and $u$ harmonic in $D$ is contained in the null-space and is of finite
codimension in the null-space.

In this section we give a careful proof that these operators are not merely
symmetric, but actually self-adjoint, and have a compact resolvent on the
orthocomplement of the null-space. For $\lambda\in \Lambda_H^1(\partial D),$ let
$\cD(\lambda)\subset\cE_1(D)$ denote the set of 1-forms such that
\begin{equation}
d\bxi\in L^2(D,\Lambda^2), \quad d_{\partial D}\bxi_t=0, \quad \text{
  and } [\bxi_t]\in \lambda.
\end{equation}
 These properties imply that 
$\bxi_t\in W^{-1/2,2}(\partial D;\Lambda^1)$ is well defined. For later applications we prove:
\begin{lemma}\label{lem7.2.008} For each $\lambda\in\Lambda_H^1(\partial D)$ the subspace $\cD(\lambda)$
is a dense subspace of $\cE_1(D).$ 
\end{lemma}
\begin{proof}
  The Hilbert space, $\cE_1(D),$ is defined as the $L^2$-closure of divergence free $\CI$
  1-forms defined in $D.$ We can  assume that $\bxi$ is smooth, and therefore
  $d\bxi\in L^2(D).$ To prove the lemma, we show that,
  given $\epsilon>0,$ there is a divergence 1-form $\balpha$ so that
  $(\bxi_t+\balpha_t)=0,$ and
  \begin{equation}
    \|\balpha\|_{L^2(D)}\leq \epsilon.
  \end{equation}
It is clear that $\bxi+\balpha\in\cD(\lambda),$ demonstrating the $\cD(\lambda)$
is dense in $\cE_1(D).$ To begin we set $\alpha'=-\bxi_t.$
 
Let $r$ denote the signed distance from $\pa D,$ normalized to be negative in
$D.$ This is known to be a smooth function in a neighborhood of $\pa D.$ For
$\delta>0$ set
  \begin{equation}
    S_{\delta}=\{-\delta\leq r\leq 0\}.
  \end{equation}
There exists $0<\delta_0$ such that each
point in $S_{\delta_0}$ has a unique point closest to $\pa D;$ we let 
\begin{equation}
  \Pi:S_{\delta_0}\longrightarrow \pa D,
\end{equation}
be the map to that closest point. Reducing $\delta_0$ if necessary, we can
assume that the lines $\{\Pi^{-1}(p):p\in \pa D\}$ smoothly foliate
$S_{\delta_0}.$ To complete the proof we show that, for any
$0<\delta<\delta_0,$ there is a divergence free 1-form, $\balpha$ supported in
$S_{\delta}$ extending $\alpha',$ for which $\|\balpha\|_{\cC^0(D)}$ is bounded
by a constant, independent of $\delta,$ times $\|\alpha'\|_{\cC^1(\pa D)}.$

For $0<\delta<\delta_0,$ choose a function $\psi\in \CI_c((-\delta,0])$ such
that $\psi(0)=1.$ In $S_{\delta}$ we set
\begin{equation}
  \balpha=\psi(r)\Pi^*\alpha'+\alpha_3 dr,
\end{equation}
which clearly extends $\alpha'.$ We need to choose $\alpha_3$ so
that
\begin{equation}
  d^*\balpha=0 \quad \text{
  and } \quad \alpha_3\restrictedto_{\{r=-\delta\}}=0.
\end{equation}
As $r$ is the geodesic distance to $\pa D$ it follows that
\begin{equation}
  d^*[\psi(r)\Pi^*\alpha']=\psi(r)d^*[\Pi^*\alpha']=\psi(r)W_1,
\end{equation}
where $W_1$ is a smooth function, independent of $\delta,$ with
\begin{equation}
  \|W_1\|_{\cC^0(S_{\delta_0})}\leq C\|\alpha'\|_{\cC^1(\pa D)}.
\end{equation}
Moreover
\begin{equation}
  d^*[\alpha_3 dr]=-N\alpha_3 +W_2\alpha_3,
\end{equation}
with $W_2$ is a smooth function independent of $\delta,$ and $N$ is the vector
field normal to the level sets of $r,$ for which $dr(N)=1,$ i.e., $N=\pa_r.$

For each $p\in \pa D,$ the linear ODE
\begin{equation}
  N\alpha_3 -W_2\alpha_3=\psi(r)W_1,
\end{equation}
with the initial condition $\alpha_3(-\delta,p)=0$ is easily solved on the
interval $[-\delta,0].$  It is clear that there is a constant, $C,$ independent of
$\delta$ so that
\begin{equation}
 \|\alpha_{3}\|_{\cC^0(S_{\delta})}\leq C\|\alpha'\|_{\cC^1(\pa D)}.
    \end{equation}
If we simply extend $\balpha$ to $D\setminus S_{\delta}$ by zero, then it is
elementary to see that the extended form is smooth and divergence free. A
simple calculation shows that if $|S_{\delta}|$ is the volume of $S_{\delta},$ then
\begin{equation}
  \|\balpha\|_{L^2(D)}\leq
  |S_{\delta}|^{\frac 12}[\|\alpha_{3}\|_{\cC^0(S_{\delta})}+\|\alpha'\|_{\cC^0(\pa D)}]
\leq (1+C)|S_{\delta}|^{\frac 12}\|\alpha'\|_{\cC^1(\pa D)}.
\end{equation}
By choosing $0<\delta$  sufficiently small this can be made smaller than $\epsilon.$

\end{proof}

\noindent
We now turn to the self-adjointness of the unbounded operators $(\cB,\cD(\lambda)).$
\begin{theorem}\label{thm7.2.001} 
For each $\lambda\in \Lambda_H^1(\partial D)$, the unbounded operator
$\cB$ with dense domain $\cD(\lambda)\subset\cE_1(D)$ is self-adjoint. It
has an infinite dimensional null-space. Its restriction to
$[\Ker(\cB,\cD(\lambda))]^{\bot}$ has a compact resolvent. For each
$\lambda\in\Lambda_H^1(\partial D)$ this restriction has a discrete
spectrum $\{k_j(\lambda)\}$ tending to $\pm\infty$. The eigenforms
$\{\bxi_j(\lambda)\}$ corresponding to non-zero eigenvalues are smooth
in the closure of $D$ and satisfy the boundary condition
  \begin{equation}
    i_{\bn}\bxi_j(\lambda)=0.
  \end{equation}
\end{theorem}
\begin{remark} Here and in the sequel we use the notation $(\cB,\cD(\lambda))$
  to denote the operator $\cB$ acting, distributionally, on elements of the
  domain $\cD(\lambda).$
\end{remark}
This theorem, with a rather different proof, appears
in~\cite{HKTSACurlBD}. We give a simple, self-contained proof of this
theorem that relies only on standard Hodge theory and
Lemma~\ref{lem7.0.002}. Later in this section we indicate how to
generalize this result to the full Maxwell system.

There is a distinguished element $\lambda_D$ of $\Lambda_H^1(\partial
D)$ consisting of the image of $H^1_{\dR}(D)\hookrightarrow
H^1_{\dR}(\partial D)$. For this choice of Lagrangian
subspace let $\{k_j(\lambda_D),\bxi_j(\lambda_D)\}$ be
an enumeration of the eigenpairs. As above we let $\{A_j:\: j=1,\dots,p\}$ be a basis for the
cycles on $\partial D$ that bound chains, $\{S_j\}$, in $D$, i.e.,
$\partial S_j=A_j$. A closed 1-form $\balpha$ on $\partial D$ belongs to
$\lambda_D$ if and only if
\begin{equation}
  \int\limits_{A_j}\balpha=0 \quad \text{ for }j=1,\dots,p.
\end{equation}
\begin{theorem}\label{thm7.3.002} 
Let $\bxi$ be an eigenform of $(\cB,\cD(\lambda_D))$ with
  eigenvalue $k\neq 0$, then $\bxi_n=0,$ and for $1\leq j\leq p$ we have that 
  \begin{equation}\label{eqn7.20.001}
    \int\limits_{S_j}\star_3\bxi=0.
  \end{equation}
Moreover, eigenforms of $\cB$ with $\bxi_n=0$ that
satisfy~\eqref{eqn7.20.001} belong to $\cD(\lambda_D)$.
\end{theorem}
\begin{remark} 
These eigenforms might be called ``zero-flux, constant-$k,$ force-free Beltrami
fields''. This result generalizes Kress's Theorem 2.5, in~\cite{kress2}, which assumes that $D$
has genus 1. 
\end{remark}
\begin{proof}[Proof of Theorem~\ref{thm7.3.002}] 
The eigenform satisfies the equation
  \begin{equation}
    d\bxi=k\star_3\bxi
  \end{equation}
and therefore, as $k\neq 0$,
\begin{equation}\label{eqn7.22.001}
  \int\limits_{S_j}\star_3\bxi=\frac{1}{k}\int\limits_{A_j}\bxi_t.
\end{equation}
By assumption $\bxi_t$ is a closed form; the eigenvalue equation then
implies that $\bxi_n=0$. As $\bxi_t=df+\omega_H$, where
$\omega_H\in\lambda_D$, there exists a closed 1-form $\balpha$ defined
in $D$ for which
\begin{equation}
  \bxi_t=\balpha_t.
\end{equation}
Therefore
\begin{equation}
  \int\limits_{A_j}\bxi_t=\int\limits_{A_j}\balpha_t=\int\limits_{S_j}d\balpha=0.
\end{equation}

It is not hard to see that the converse statement holds as well. If $\bxi$ is a
divergence-free 1-form with $\bxi_n=0$ satisfying $\cB\bxi=k\star_3\bxi$, and
the vanishing flux conditions in~\eqref{eqn7.20.001}, then $\bxi$ belongs to
$\cD(\lambda_D)$.  From the equation we conclude that $\bxi_t$ is closed and
therefore
\begin{equation}
  \bxi_t=df+\omega_H.
\end{equation}
The integration by parts argument giving~\eqref{eqn7.22.001} shows that
for each $j = 1, \ldots, p$, 
\begin{equation}
  \int\limits_{A_j}\omega_H=0.
\end{equation}
The assertion follows from the fact that these $p$ linear conditions define the
image of $H^1_{\dR}(D)$ in $H^1_{\dR}(\partial D)$.
\end{proof}

Now we turn to the proof of Theorem~\ref{thm7.2.001} for which we use the following
lemma.
\begin{lemma}\label{lem7.5} 
Let $(X,\|\cdot\|_X)$ and $(Y,\|\cdot\|_Y)$ be Hilbert spaces with
$Y\subset X$. For each $y\in Y$, suppose that we have the estimate
\begin{equation}\label{eqn7.27.001}
   \|y\|_X\leq\|y\|_Y.
\end{equation}
 Let $L$ be a closed, symmetric operator defined on a dense domain $\mathcal D(L)$,
\[
Y\subset \mathcal D(L)\subset X.
\]
Let $\pi_0$ be the $X$-orthogonal projection onto
 $\Ker L$, and $\pi_1=\Id-\pi_0$. Suppose that there is an operator $Q$,
\[
Q:[\Ker L]^{\bot}\to Y\cap [\Ker L]^{\bot}
\]
with $LQ\pi_1=\pi_1$. Suppose further that $Q$ satisfies an estimate
of the form:
\[
\|Q\pi_1 x\|_Y\leq C\|x\|_X.
\]
There is an $\epsilon>0$ so that the operators $(L-\mu)$ are
invertible for $0<|\mu|<\epsilon$, and therefore $(L,\mathcal D(L))$ is
self-adjoint.
\end{lemma}
\begin{proof} 
Formally it is easy to see that $(L-\mu)^{-1}$ is given by
  \begin{equation}\label{7.28.001}
    R_{\mu} x=Q\left[\sum_{j=0}^{\infty}(\mu Q)^j\pi_1
      x\right]-\frac{\pi_0 x}{\mu}.
  \end{equation}
We need to show that this series converges (in an appropriate sense) and
takes values in $\mathcal D(L)$. The term $\frac{\pi_0 x}{\mu}$ obviously
belongs to $\mathcal D(L)$.  By assumption, there is a constant $C$ so that
\begin{equation}
  \|Q\pi_1x\|_Y\leq C\|x\|_X.
\end{equation}
As $\pi_0Q=0$, using~\eqref{eqn7.27.001} we can iterate this estimate
to conclude that
\begin{equation}
  \|Q^j\pi_1 x\|_Y\leq C^j\|x\|_X.
\end{equation}
This allows us to conclude that for $\mu<C^{-1}$ the series
in~\eqref{7.28.001} converges in $Y$, which implies that this term is
also in $\mathcal D(L)$. Applying $L-\mu$ to this sum we get
\begin{equation}\label{eqn7.40.008}
\begin{split}
  (L-\mu)R_{\mu}x&=\sum_{j=0}^{\infty}(\mu Q)^j\pi_1 x-\mu
  Q\sum_{j=0}^{\infty}(\mu Q)^j\pi_1 x+\pi_0x\\
&=\pi_1x+\pi_0x=x.
\end{split}
\end{equation}
If $0<|\mu|<\|Q\|^{-1}$, then $L-\mu$ is surjective. Since $L-\mu$ is symmetric
for real $\mu,$ equation~\eqref{eqn7.40.008} implies that $(L-\mu)$ is
injective.  Theorem 13.11(d) in~\cite{RudinFA} shows that $(L,\mathcal D(L))$
is a self-adjoint operator, thereby completing the proof of the lemma.
\end{proof}

\begin{proof}[Proof of Theorem~\ref{thm7.2.001}]
To apply this to the operator $(\cB,\cD(\lambda))$ we need to identify
$Y$ and the operator $Q$.  The space $Y$ is the subspace of
$W^{1,2}(D;\Lambda^1)$ consisting of forms that satisfy the conditions
\begin{equation}
  d^*\bxi=0 \quad \text{and} \quad d_{\partial D}\bxi_t=0
\end{equation}
with $[\bxi_t]\in \lambda$.  For the norm, we use the
$W^{1,2}(D;\Lambda^1)$-norm which obviously satisfies
\begin{equation}
  \|\bxi\|_{L^2}\leq \|\bxi\|_{W^{1,2}}.
\end{equation}
It is  clear that $Y\subset\cD(\lambda)$. 

Let $\{\psi_1,\dots,\psi_p\}$ denote an orthonormal basis of
$\cH^1(D)$. These forms satisfy
\begin{equation}
\begin{split}
  d\psi_j=d^*\psi_j &=0, \\
  i_{\bn}\psi_j&=0 \qquad \text{on } \partial D,
\end{split}
\end{equation}
and
\begin{equation}
  \int\limits_{D}\psi_i\wedge\star_3\psi_j=\delta_{ij}.
\end{equation}
We now need to characterize $\Ker(\cB,\cD(\lambda))$. These are
$L^2$ 1-forms, $\bxi$, that satisfy the equations $d^*\bxi=d\bxi=0$ and
the boundary condition $[\bxi_t]\in\lambda$. Each such form has a
decomposition
\begin{equation}
  \bxi=du+\sum_{j=1}^p\alpha_j\psi_j.
\end{equation}
Here $u$ is a harmonic function determined by the condition
\begin{equation}
  i_{\bn}\bxi=\frac{\pa u}{\pa\bn}.
\end{equation}
Observe that  as $d^*\bxi=0$, we see that
\begin{equation}
  \int\limits_{\partial D}i_{\bn}\bxi dA=\int\limits_{\partial
    D}\star_3\bxi=\int\limits_{D}d\star_3\bxi=0,
\end{equation}
and therefore $u$ exists. Furthermore, Lemma~\ref{lem7.0.002} and the closed
graph theorem show that there are constants $C$ and $C'$ so that
\begin{equation}
  \|u\|_{W^{1,2}(D)}\leq C\|i_{\bn}\bxi\|_{H^{-\frac 12}(\partial
    D)}\leq C'\|\bxi\|_{L^2(D)}.
\end{equation}
The coefficients $\{\alpha_j\}$ are given by
\begin{equation}
  \alpha_j=\langle\bxi-du,\psi_j\rangle.
\end{equation}

The boundary condition defining $\cD(\lambda)$ requires that 
\begin{equation}
  \left[\sum_{j=1}^p\alpha_j[\psi_{jt}]\right]\in\lambda_D\cap\lambda,
\end{equation}
where as before $\lambda_D$ is the Lagrangian subspace determined by
$H^1_{\dR}(D)\hookrightarrow H^1_{\dR}(\partial D)$. Generically, this
intersection is just $\{0\}$. Letting $q=\dim \lambda_D\cap\lambda$,
we can  normalize the basis of harmonic 1-forms on $D$ so that
$\{\psi_1,\dots,\psi_q\}$ spans
$\Ker(\cB,\cD(\lambda))\cap\cH^1(D)$. This means that
$\{[\psi_{1t}],\dots,[\psi_{qt}]\}$ span $\lambda_D\cap\lambda$, and
\begin{equation}
  \int\limits_{D}\psi_i\wedge\star_3\psi_j=0
\end{equation}
for $1\leq i\leq q$ and $q+1\leq j\leq p$.  With these normalizations,
we have
\begin{equation}
  \Ker(\cB,\cD(\lambda))= \left\{ du+\sum_{j=1}^q\gamma_j\psi_j : u
  \in W^{1,2}(D), \, u \text{ harmonic}, \, \gamma_j\in\bbR \right\}.
\end{equation}

If a 1-form $\bEta\in\cE_1(D)$ is in $[\Ker(\cB,\cD(\lambda))]^{\bot}$, then
for every harmonic function $u\in W^{1,2}(D),$ we have:
\begin{equation}
\begin{split}
  0=\langle du,\bEta\rangle_D&=\langle u,d^*\bEta\rangle_D+\langle
  u,i_{\bn}\bEta\rangle_{\partial D}\\
&=\langle u,i_{\bn}\bEta\rangle_{\partial D}.
\end{split}
\end{equation}
Lemma~\ref{lem7.0.002} implies that this calculation makes sense as
$u\restrictedto_{\pa D}\in W^{1/2,2}(\pa D)$ and $i_{\bn}\bEta\in
W^{-1/2,2}(\pa D).$ As $u\restrictedto_{\pa D}$ is an arbitrary element of
$W^{1/2,2}(\pa D)$ we see that $i_{\bn}\bEta=0$. We also must require that
\begin{equation}
  \langle \psi_j,\bEta\rangle_{D}=0\quad \text{ for }j=1,\dots,q.
\end{equation}

We now construct the solution operator required by the lemma for
$\bEta$ satisfying these conditions. Such a 1-form can be written as
\begin{equation}\label{eqn7.43.003}
  \bEta=\bEta_0+\sum_{j=q+1}^{p}\alpha_j\psi_j,
\end{equation}
where $i_{\bn}\bEta_0=0$ and $\bEta_0$ is orthogonal to $\cH^1(D)$. We would
like to solve
\begin{equation}
  d\omega=\star_3\bEta_0+\sum_{j=q+1}^{p}\alpha_j\star_3\psi_j,
\end{equation}
with $\omega\in Y$.  We begin with some existence results that are
independent of the choice of $\lambda$.
\begin{lemma}
  If $\eta_0$ is a divergence-free 1-form in $D$ orthogonal to
  $\cH^1(D)$ for which $i_{\bn}\eta_0=0$, then there exists a
  divergence-free 1-form $\omega_0$ orthogonal to $\cH^1(D)$
  satisfying:
  \begin{equation}\label{eqn7.45.0004}
    \begin{split}
      \star_3d\omega_0&=\eta_0,\\ 
      d_{\partial D}\omega_{0t}&=0,\\
      i_{\bn}\omega_0&=0.
    \end{split}
  \end{equation}
Moreover, there is a constant $C$ such that
  \begin{equation}\label{eqn7.50.007}
      \|\omega_0\|_{W^{1,2}(D)}\leq C\|\bEta_0\|_{L^2(D)}.
  \end{equation}
\end{lemma}
\begin{proof}
  Observe that the conditions in~\eqref{eqn7.45.0004} along
  with~\eqref{eqn7.8.007} imply the estimate in~\eqref{eqn7.50.007}.
The 2-form $\star_3\bEta_0$ is closed and has a vanishing restriction
to $\partial D;$ it is also orthogonal to $\star_3\cH^1(D)$, which is
the null-space of the Dirichlet Laplace operator on 2-forms. If
$R_{2t}$ denotes the partial inverse of this Laplace operator taking
values in the orthogonal complement of the null-space, then we set
\begin{equation}
  \tomega_0=d^*R_{2t}\star_3\bEta_0.
\end{equation}
The properties of this operator imply that 
\begin{equation}\label{eqn7.48.0004}
 d\tomega_0=\star_3\bEta_0 \quad \text{ and } \quad \tomega_{0t}=0.
\end{equation}
There is, moreover, a constant $C$ such that
\begin{equation}\label{eqn7.49.0004}
  \|\tomega_0\|_{W^{1,2}(D)}\leq C\|\bEta_0\|_{L^2(D)}\leq C\|\bEta\|_{L^2(D)}.
\end{equation}

As $\star_3\tomega_0=d\star_3R_{2t}\star_3\bEta_0$, Stokes' theorem implies that
\begin{equation}
  \int\limits_{\partial D}i_{\bn}\tomega_0=\int\limits_{\partial
    D}\star_3\tomega_0=0.
\end{equation}
Thus we find a harmonic function $v_0$ so that
\begin{equation}
  i_{\bn}\tomega_0=\pa_{\bn}v_0.
\end{equation}
Standard estimates imply that there are constants $C$ and $C'$ so that
\begin{equation}\label{eqn7.52.0004}
  \|v_0\|_{H^2(D)}\leq C\|i_{\bn}\tomega_0\|_{H^{\frac 12}(\partial
    D)}\leq C'\|\tomega_0\|_{H^{1}(D)}.
\end{equation}
The 1-form $\omega'_0=\tomega_0-dv_0$ satisfies the conditions
in~\eqref{eqn7.45.0004}. To obtain $\omega_0$, we add an element of
$\cH^1(D)$ to $\omega'_0$ to get a solution to~\eqref{eqn7.45.0004}
which is orthogonal to $\cH^1(D)$. The fact that $d_{\partial
  D}\omega_{0t}=0$ follows from~\eqref{eqn7.45.0004} and the fact that
$i_{\bn}\eta_0=0$, or from~\eqref{eqn7.48.0004}. The estimate for
$\|\omega_0\|_{W^{1,2}}$ follows from~\eqref{eqn7.49.0004}
and~\eqref{eqn7.52.0004}.
\end{proof}

For the next step in the construction we use the following lemma:
\begin{lemma} 
For $1\leq j\leq p$, there are $\CI$ 1-forms
$\{\mu_{j0}\}$ orthogonal to $\cH^1(D)$ so that
\begin{equation}
\begin{split}
  d\mu_{j0}&=\star_3\psi_j,\\
  d^*\mu_{j0}&=0,\\
  i_{\bn}\mu_{j0}&=0.
\end{split}
\end{equation}

\end{lemma}
\begin{proof}
Let $r$ be a defining function for $\partial D$; this is a function
defined in a neighborhood, $U$, of $\partial D$ so that $D\cap
U=\{r<0\}$. The defining function $r$ is normalized so that
$i_{\bn}dr=1$ along $\partial D$. We also choose $\epsilon>0$, and a
function $\varphi\in\CI_c((-\epsilon,\epsilon))$ which satisfies
$\varphi(0)=1$.

Let $\balpha'_j=i_{\bn}\star_3\psi_j;$ this is a smooth 1-form on
$\partial D$. Let $\balpha_j$ be a smooth extension of $\balpha'_j$ to
a neighborhood of $\partial D$.  If $\epsilon>0$ is chosen small
enough, then
  \begin{equation}
    \chi_j=r\varphi(r)\balpha_j
  \end{equation}
is a smooth 1-form in $D$ that vanishes on $\partial D$ and satisfies
\begin{equation}
  i_{\bn}d\chi_j=\balpha'_j.
\end{equation}
Thus $\star_3\psi_j-d\chi_j$ is a smooth 2-form with vanishing normal
component. As $H^2_{\dR}(D)=0$, the Laplacian on 2-forms satisfying the Neumann
boundary condition $i_{\bn}\theta=0$ is invertible. Let $R_{2n}$ denote its
inverse and set
\begin{equation}
  \tmu_j=d^*R_{2n}[\star_3\psi_j-d\chi_j].
\end{equation}
This gives a smooth solution to the equation $d(\tmu_j+\chi_j)=\star_3\psi_j;$
it is also  clear that $i_{\bn}\tmu_j=0$ and $d^*\tmu_j=0$.

We need to modify $\chi_j$ to make it divergence-free; this requires a smooth function
$u_j$ that vanishes on $\partial D$ and satisfies
\begin{equation}
  d^*du_j=d^*\chi_j\quad \text{ and } \quad \pa_{\bn}u_j=0.
\end{equation}
The range of the Neumann Laplacian on scalar functions consists of functions of mean
zero. We observe that
\begin{equation}
  \int\limits_{D}\star_3d^*\chi_j=\int\limits_{D}d\star\chi_j=0.
\end{equation}
The last equality follows from Stokes' theorem and the fact that $\chi_j$
vanishes on $\partial D$.
 Thus $\mu'_{j}=\tmu_j+\chi_j-du_j$ is a divergence-free solution to 
\begin{equation}
  d\mu'_{j}=\star_3\psi_j\quad \text{ with } \quad i_{\bn}\mu'_{j}=0.
\end{equation}
As $[\star_3\psi_j]_t=0$ we see that $d_{\partial D}\mu'_{jt}=0$. We are free to add an
element $\varphi_j$ of $\cH^1(D)$ to $\mu'_{j}$ so that
$\mu_{j0}=\mu'_{j}+\varphi_j$ is orthogonal to $\cH^1(D)$. 
\end{proof}

We have found a divergence-free solution $Q_0\eta$,
\begin{equation}
  Q_0\eta=\omega_0+\sum_{j=q+1}^p\alpha_j\mu_{j0},
\end{equation}
to the equation $\star_3dQ_0\eta=\eta$ which is orthogonal to
$\cH^1(D)$ with $i_{\bn}Q_0\eta=0$. This easily implies that $Q_0\eta$
is orthogonal to the range of $d$, and therefore to
$\Ker(\cB,\cD(\lambda))$. As before the fact that $i_{\bn}\eta=0$
shows that $d_{\partial D}[Q_0\eta]_t=0$. All that remains is to
modify the solution so that the cohomology class of the boundary data
lies in $\lambda$.

Our final solution takes the form
\begin{equation}
  Q\eta=Q_0\eta+\sum_{k=q+1}^p\cM_{k}(\lambda)\psi_k,
\end{equation}
where the coefficients $\{\cM_{k}(\lambda)\}$ are chosen so that
$[(Q\eta)_t]\in\lambda$. We can choose independent linear functionals
$\{\ell_1,\dots,\ell_p\}$ on $H^1_{\dR}(\partial D)$ so that
\begin{equation}
  \lambda=\bigcap\limits_{i=1}^p\Ker \ell_j,
\end{equation}
and $\lambda_D\subset \Ker \ell_i$ for $i=1,\dots,q$. Since
$\{[\psi_{1t}],\dots,[\psi_{qt}]\}$ is a basis for $\lambda\cap\lambda_D$, and
both $\lambda$ and $\lambda_D$ are Lagrangian subspaces, it is clear that 
for $1\leq i\leq q$, the functionals
\begin{equation}
  \ell_i(\omega)=\int\limits_{\partial D}\omega\wedge \psi_{i}
\end{equation}
span $\lambda^{\bot}\cap\lambda_D^{\bot}$. Moreover, the matrix
\begin{equation}
  \big[\ell_i([\psi_{kt}])\big]_{q+1\leq i,k\leq p}
\end{equation}
has rank $p-q$.

For $1\leq i\leq q$, Stokes' theorem and the fact that
$\star_3dQ_0\eta=\eta$ show that
\begin{equation}
\begin{split}
  \ell_i((Q_0\eta)_t)=&\int\limits_{\partial D}Q_0\eta\wedge \psi_{i}\\
=&\int\limits_{D}\star_3\eta\wedge \psi_{i}\\
=&0.
\end{split}
\end{equation}
In order for $[(Q\eta)_t]\in\lambda$, we require coefficients
$\{\cM_{k}(\lambda)\}$ so that 
\begin{equation}
  \ell_i([(Q_0\eta)_t])=-\sum_{n=q+1}^p\cM_{k}(\lambda)\ell_i(\psi_k)\quad
  \text{ for } q+1 \leq i\leq p.
\end{equation}
As the matrix $[\ell_i(\psi_k)]_{q+1\leq i,k\leq p}$ has maximal rank, this
system of equations is uniquely solvable. 

Using these results and Lemma~\ref{lem7.5} we complete the proof of
Theorem~\ref{thm7.2.001}.  We have verified that $(\cB,\cD(\lambda))$
has an infinite dimensional null-space. Since this operator is
symmetric and densely defined, in order to prove its
self-adjointness it follows from Theorem 13.11(d) in~\cite{RudinFA}
that it suffices to prove that for some real $\theta$ the operator
$$(\cB-\theta\Id):\cD(\lambda)\to \cE_1(D)$$ 
is invertible. Given our construction of $Q$, Lemma~\ref{lem7.5} shows that
this holds for all sufficiently small, non-zero complex numbers $\theta$.

The fact that $\cB\restrictedto_{[\Ker(\cB,\cD(\lambda))]^{\bot}}$ has a
compact resolvent follows immediately from the fact that 
\begin{equation}
  Q:[\Ker(\cB,\cD(\lambda))]^{\bot}\to W^{1,2}(D;\Lambda^1).
\end{equation} 
In the course of constructing $Q$ we showed that $i_{\bn}\bEta=0$, for
$\bEta\in\cD(\lambda)\cap [\Ker(\bB,\cD(\lambda)]^{\bot}$, so this certainly
holds for the eigenforms $\{\bxi_j(\lambda)\}$. The fact that these forms are
smooth on $\overline{D}$ now follows from a standard bootstrap argument
using~\eqref{eqn7.9.007}, and that facts that, for eigenfunctions orthogonal to
the null-space we have:
\begin{equation}
\Delta\bxi=\cB^2\bxi,\, d\bxi=\mu\star_3\bxi\quad \text{ and } \quad
i_{\bn}\bxi=0.
\end{equation}
\end{proof}

\begin{remark} The partial inverse $Q_{\lambda}$ has a very simple form:
  \begin{equation}
    Q_{\lambda}=\left(\Id+\sum_{j,k=q+1}^p\cN_{jk}(\lambda)\psi_j(x)\otimes
      \ell_{k}\circ\rho_t\right)\circ Q_0,
  \end{equation}
where $\rho_t(\eta)=\eta_t$. Here $\{\ell_1,\dots,\ell_p\}$ are the linear
functionals defining $\lambda$. Note that if $\lambda_1$ and $\lambda_2$ are
two Lagrangian subspaces with trivial  intersection with $\lambda_D$, then the
difference of the partial inverses is the finite rank operator:
\begin{equation}\label{eqn7.74.007}
  Q_{\lambda_1}-Q_{\lambda_2}
  =\sum_{j,k=1}^p\left[\cN_{jk}(\lambda_1)\psi_j(x)\otimes
    \ell_{1k}-\cN_{jk}(\lambda_2)\psi_j(x)\otimes
    \ell_{2k}\right]\circ\rho_t\circ Q_0.
\end{equation}
It is also notable that, in this case, the Hilbert space
$H_0=[\Ker(\cB,\cD(\lambda))]^{\bot}$ does not depend on $\lambda$. We
let $\cB^0_{\lambda}$ denote the restriction of $(\cB,\cD(\lambda))$
to $H_0;$ of course $Q_{\lambda}$ is its inverse. The resolvents
satisfy the identity
\begin{equation}
  \mu (\cB^0_{\lambda}-\mu)^{-1}=Q_{\lambda}\left(\frac{\Id}{\mu}-
  Q_{\lambda}\right)^{-1}.
\end{equation}
\end{remark}

\subsection{Finding Force-Free Beltrami Eigenfields}\label{sec7.2}
In the previous section we demonstrated that for each $\lambda\in\Lambda_H^1(D)$
there is a countable sequence $\{k_j(\lambda)\}$ and 1-forms
$\{\bxi_j(\lambda)\}$ that satisfy the conditions
\begin{equation}
\begin{aligned}
  d\bxi_j(\lambda)&=k_j(\lambda)\star_3 \bxi_j(\lambda), & \quad
  d^*\bxi_j(\lambda)&=0, \\ d_{\partial D}[\bxi_j(\lambda)]_t&=0, &
  \quad [\bxi_j(\lambda)_t]&\in\lambda.
\end{aligned}
\end{equation}
As noted, the equations imply that $i_{\bn}\bxi_j(\lambda)=0$ as well. The
question then arises how to find these 1-forms. The Debye source representation
provides a complete solution to this problem, and reduces it to a question of
finding frequencies for which a system of Fredholm equations of second kind on
$\pa D$ has a non-trivial null-space and then finding the null-vectors.

In Section~\ref{sec6.0}, we showed that solutions of the \THME[$k$] that also
satisfy $\star_3\bxi=i\bEta$ are specified by Debye source data satisfying the
conditions
\begin{equation}
\begin{split}
  q&=-ir, \\ 
  \star_2\bj_{H0}&=-i\bj_{H0},\\
  \star_2\bj_{Hl}&=i\bj_{Hl},
\end{split}
\end{equation}
for $l=1,\dots,d$.
We let $\cN_{\bxi}(r,\bj_H)$ denote the Debye source
operator for the normal component of the electric field acting on data
satisfying these relations.  For each $\lambda\in\Lambda_H^1(\partial D)$ there
is a collection of 1-cycles
$$\{C_j(\lambda):\: j=1,\dots,p\}\subset H_1(\partial D;\bbR)$$ 
such that a closed form $[\balpha]\in\lambda$ if and only if
\begin{equation}
  \int\limits_{C_j(\lambda)}\balpha=0.
\end{equation}
This follows from the fact that $[H^1_{\dR}(\partial D)]'\simeq
H_1(\partial D;\bbR)$.

We recall that the Debye source representation is injective. A
frequency $k_0$ is therefore a non-zero eigenvalue of
$(\bB,\cD(\lambda))$ if and only $k_0$ is a frequency for which the
system of equations
\begin{equation}\label{eqn7.72.003}
  \begin{split}
    \cN_{\bxi}(r,\bj_H)&=f,\\
\int\limits_{C_j(\lambda)}\cT_{\bxi}(r,\bj_H)&=a_j, \qquad j=1,\ldots,p,
  \end{split}
\end{equation}
has a non-trivial null-space.  Recall that $f$ must
have mean zero on every component of $\partial D$. The dimension of
this null-space is exactly equal to the multiplicity of $k_0$ as an
eigenvalue of $(\bB,\cD(\lambda))$.

As noted above, in the case of $\lambda=\lambda_D$, the cycles
$\{C_j(\lambda_D)\}$ can be taken to be the cycles $\{A_j\}$ that bound chains
$\{S_j\}\subset D$. Hence, to find the classical  constant-$k$, zero flux,
force-free Beltrami fields in a domain $D$, we need to find the frequencies
$\{k_l\}$ for which~\eqref{eqn7.72.003} with $\{C_j(\lambda)\}=\{A_j\}$
has a non-trivial null-space. If $\bxi$ is such a field, then
\begin{equation}
d\bxi=k\star_3\bxi, \qquad \bxi_ndA_{\partial D}=d\bxi_t=0,
\end{equation}
and
\begin{equation}
  \int\limits_{A_j}\bxi=k\int\limits_{S_j}\star_3\bxi = 0.
\end{equation}
Thus the eigenfields of $(\cB,\cD(\lambda_D))$ are constant-$k$, force-free
Beltrami fields with vanishing flux.

Theorems~\ref{prop3.5.01} and~\ref{thm7.2.001}, and the argument above have
an interesting corollary. We let $\sigma(\cB,\lambda)$ denote the non-zero
spectrum of the operator $(\cB,\cD(\lambda))$, and define the set
\begin{equation}
  E_{\cB}=\bigcap\limits_{\lambda\in \Lambda_H^1(\partial D)}\sigma(\cB,\lambda).
\end{equation}
\begin{corollary}\label{cor7.7.003} 
For $k\in \mathbb C^+\setminus E_{\cB}$ the space of solutions to the
Beltrami equation
  \begin{equation}\label{eqn7.76.003}
    \cB\bxi=k\bxi\quad \text{ with }\quad i_{\bn}\bxi=0
  \end{equation}
has dimension $p$, where $p$ is the total genus of $\partial D$.
\end{corollary}
\begin{proof} 
As noted above, for each $\lambda\in \Lambda_H^1(\partial D)$ there is a
collection of 1-cycles on $\partial D$,
$\{C_1(\lambda),\dots,C_p(\lambda)\}$, so that a closed form $\alpha$
satisfies $[\alpha]\in\lambda$ if and only if for $j=1,\dots,p$,
  \begin{equation}
    \int\limits_{C_j(\lambda)}\alpha=0.
  \end{equation}
  The system of equations in~\eqref{eqn7.72.003} is a Fredholm system
  of index zero. Theorem~\eqref{prop3.5.01} implies that this system
  is invertible for any $k\notin \sigma(\cB,\lambda)$, and therefore
  the space of solutions to~\eqref{eqn7.76.003} is $p$ dimensional for
  such $k$. As this holds for any choice of $\lambda\in
  \Lambda_H^1(\partial D)$, the corollary follows.
\end{proof}

It seems a very interesting question whether or not $E_{\cB}=\emptyset$. If
this is so, then the $\dim \cN^{1,0}_{k}(D)=p$ for any $k\in \mathbb C^+$. If
$E_{\cB}\neq\emptyset$, then this set is a new spectral invariant of the
embedding of $\partial D$ into $\bbR^3$. As we see in the next section, it can
happen that $E_{\cB}\neq\emptyset$ if $D$ has a continuous symmetry.

As noted above, the spectral theory of the operators
$(\cB,\cD(\lambda))$ defines a map from $\Lambda_H^1(\partial D)$ to
the eigendata $\{(k_j(\lambda),\bxi_j(\lambda))\}$. The dependence of
each eigenvalue on the choice of $\lambda$ is something that is
readily investigated. For each $\lambda\in \Lambda_H^1(\partial D)$ and
$\mu\in\bbR$ we let
\begin{equation}
  d_{\lambda}(\mu)=\dim\Ker((\cB-\mu\Id,\cD(\lambda)),
\end{equation}
and set
\begin{equation}
  d_{\min}(\mu)=\min\{d_{\lambda}(\mu):\: \lambda\in \Lambda_H^1(\partial D)\}.
\end{equation}
If $d_{\min}(\mu)>0,$ then for every $\lambda\in\Lambda_H^1(\partial D)$ there
is subspace of $\cD(\lambda)$ of this dimension consisting of solutions to
\begin{equation}
  (\cB-\mu\Id)\bxi=0.
\end{equation}
Indeed the estimates on eigenfunctions that follow
from~\eqref{eqn7.9.007} show that if $<\lambda_j>$ converges to $\lambda,$ then 
\begin{equation}
  \limsup_{j\to\infty}d_{\lambda_j}(\mu)\leq d_{\lambda}(\mu).
\end{equation}
Hence the subset
\begin{equation}
  \cW_{\min}(\mu)=\{ \lambda\in\Lambda_H^1(\partial D):d_{\lambda}(\mu)=d_{\min}(\mu)\}
\end{equation}
is open. This shows, in particular that it has non-trivial
intersection with the dense open set 
\begin{equation}
  \cG=\{\lambda\in \Lambda_H^1(\partial D):\: \lambda\cap\lambda_{D}=\{0\}\}.
\end{equation}
With this observation and analytic perturbation theory we  prove the
following result:
\begin{theorem} If $d_{\min}(\mu)>0,$ then there is a subspace
  $\cS_{\mu}^{0}\subset\cE_1(D)$, with $\bxi\in \cS_{\mu}^{0}$ satisfying:
  \begin{equation}
    \begin{aligned}
     d\bxi &\in L^2(D;\Lambda^2), &\quad \cB\bxi&=\mu\bxi,\\ d_{\pa
       D}\bxi_t&=0, &\quad [\bxi_t]&=0 \quad \text{ in } H^1_{\dR}(\pa
     D).
    \end{aligned}
      \end{equation}
 $\cS_{\mu}^0\subset\cD(\lambda)$
for every $\lambda\in\Lambda_H^1(\partial D),$ and $\dim  \cS_{\mu}^{0}=d_{\min}(\mu).$  
\end{theorem}
\begin{proof}
  Suppose that $\lambda_s$ is an analytic curve contained in
  $\cW_{\min}(\mu)\cap \cG.$ In particular, there exist analytic families of cycles
  $\{C_l(s):\: l=1,\dots,p\}$ in $H_1(\partial D;\bbR)$ so that a
  closed 1-form $\balpha\in\lambda_s$ if and only if
\begin{equation}
  \int\limits_{C_l(s)}\balpha=0\quad \text{ for }l=1,\dots,p.
\end{equation}It is a simple consequence of~\eqref{eqn7.74.007}
that $(\cB,\cD(\lambda_s))$ is a one parameter analytic family of operators. In
this case the eigenvectors and eigenspaces can be parametrized
analytically. 

Let $\cS_{\mu}(s)$ denote the $\mu$-eigenspace of $(\cB,\cD(\lambda_s)).$ For
$j = 1,\ldots,d_{\min}(\mu)$, let $(k_j(s),\bxi_j(s))$ be  analytic families
satisfying
\begin{equation}
  \cB\bxi_j(s)=k_j(s)\bxi_j(s) \quad \text{and} \quad
  \int\limits_{C_l(s)}\bxi_j(s)=0,
\end{equation}
with $\{\bxi_1(0),\dots,\bxi_{d_{\min}(\mu)}(0)\}$ an orthonormal basis for $\cS_{\mu}(0).$
  Differentiating, and setting $s=0$ gives:
\begin{equation}
  \cB\bxi_j'(0)=k_j'(0)\bxi_j(0)+k_j(0)\bxi_j'(0),
\end{equation}
and
\begin{equation}\label{eqn7.80.001}
  \int\limits_{C'_l(0)}\bxi_j(0)+\int\limits_{C_l(0)}\bxi'_j(0)=0.
\end{equation}
If the $\bxi_j(s)$ are normalized to have norm $1$, then this implies
\begin{equation}
  k_j'(0)=\langle\cB\bxi_j'(0),\bxi_j(0)\rangle.
\end{equation}
Integrating by parts on the right hand side we see that
\begin{equation}\label{eqn7.98.007}
   k_j'(0)=\int_{\partial D}\bxi_j'(0)\wedge\bxi_j(0).
\end{equation}
Because the curve $\lambda_s\subset \cW_{\min}(\mu)\cap \cG,$ so that
$\dim\cS_{\mu}(s)=d_{\min}(\mu);$ it is immediate from the discreteness of the
spectra of $(\cB,\cD(\lambda_s))$ that 
\begin{equation}
  k_j'(s)\equiv 0.\text{ for }j=1,\dots,d_{\min}(\mu).
\end{equation}

It is easily established that there is a dual family of cycles $\{
B_1,\dots,B_p\}$, which, together with $\{C_j(0):\;j=1,\dots,p\}$, is a
generating set for $H_1(\partial D;\bbR)$. It is classical that we can choose $\{
B_1,\dots,B_p\}$ so that if $\balpha$ and
$\bBeta$ are closed 1-forms, then 
\begin{equation} 
  \int\limits_{\partial D}\balpha\wedge\bBeta=\sum_{l=1}^p
\left[\int\limits_{C_l(0)}\balpha\cdot \int\limits_{B_l}\bBeta-
\int\limits_{B_l}\balpha\cdot \int\limits_{C_l(0)}\bBeta\right],
\end{equation}
see~\cite{GriffithsHarris}. Using this formula, and the boundary condition satisfied by $\bxi_j(0)$,
we see that
\begin{equation}\label{eqn7.84.001}
  \int_{\partial D}\bxi_j'(0)\wedge\bxi_j(0)=
\sum_{l=1}^p \int\limits_{C_l(0)}\bxi'_j(0)\cdot \int\limits_{B_l}\bxi_j(0).
\end{equation}

If $\bxi_j(0)$ is not trivial in $H^1_{\dR}(\partial D)$, then there
is an $l_0$ such that
\begin{equation}
  \int\limits_{B_{l_0}}\bxi_j(0)\neq 0.
\end{equation}
This is because $\{C_l(0), B_l:\: l=1,\dots p\}$ is a basis for $H_1(\partial
D;\bbR)$.  We are free to choose the curve $\lambda_s$ so that
$\pa_s\lambda_s\restrictedto_{s=0}$ is any vector in
$T_{\lambda_0}\Lambda_H^1(\partial D).$ As $\{B_1,\dots,B_p\}$ is dual to
$\{C_1(0),\dots,C_p(0)\},$ it is not hard to see that we can choose our curve
$\lambda_s$ so that $C_l'(0)=0$, unless $l=l_0$, in which case
$C_{l_0}'(0)=B_{l_0}$. From the variational equations~\eqref{eqn7.80.001}
and~\eqref{eqn7.84.001} it would then follow that
\begin{equation}
  k_j'(0)=-\left[\int\limits_{B_{l_0}}\bxi_{j}(0)\right]^2\neq 0,
\end{equation}
which contradicts our earlier observation that $k_j'(0)=0.$ Hence, for
$j=1,\dots,d_{\min}(\mu),$  the 1-form $\bxi_{jt}(0)$
must be trivial in $H^1_{\dR}(\pa D).$

Any divergence-free 1-form $\bxi\in W^{1,2}(D),$ with $d\bxi_t=0$ and
$[\bxi_t]=0$ in $H^1_{\dR}(\pa D)$ automatically belongs to $\cD(\lambda)$ for
all $\lambda\in\Lambda_H^1(\partial D).$ This shows that  setting
$\cS_{\mu}^0=\cS_{\mu}(0)$ completes the proof of the theorem.
\end{proof}

From their definitions, it is clear that
\begin{equation}
  E_{\cB}=\{\mu:\: d_{\min}(\mu)>0\}.
\end{equation}
The theorem shows that $\mu\in E_{\cB}$ if and only there exists a non-zero
divergence free 1-form, $\bxi$ with $\cB\bxi=\mu\bxi$ and $d_{\pa D}\bxi_t=0,$
with $[\bxi_t]=0$ in $H^1_{\dR}(\pa D).$ In the next section we show that
$E_{\cB}\neq\emptyset$ for a solid torus of revolution.  In the last section we
identify $E_{\cB},$ for a round ball, with the Dirichlet spectrum of the scalar
Laplace operator. Whether $E_{\cB}=\emptyset$ for generic tori or higher genus
surfaces is far from clear. Formul{\ae}~\eqref{eqn7.80.001}
and~\eqref{eqn7.98.007} should prove useful in the analysis of this question.

\subsection{Beltrami Fields on the Torus}

For numerical examples of the discussion in the previous sections,
we compute exceptional frequencies $\{k_j\}$ and the corresponding force-free,
zero-flux Beltrami fields in the interior of a torus. 
Numerical results are provided for the unit ball in a subsequent section.
If $x,y,z$ are the usual Cartesian coordinates,
let $D$ be the bounded domain with genus~$1$ boundary $\partial D$
(a torus) given by
\begin{equation}
\label{eq-torus}
\begin{split}
x(\theta,\phi) &= (2+\cos\phi) \cos\theta, \\
y(\theta,\phi) &= (2+\cos\phi) \sin\theta, \\
z(\theta,\phi) &= \sin\phi,
\end{split}
\end{equation}
for $\theta,\phi \in [0,2\pi) \times [0,2\pi)$. Let us define a local
    orthonormal basis on $\partial D$, $( \hbphi, \hbtheta, \hbn)$,
    with
\begin{equation}
\begin{split}
\hbphi &= ( -\sin\phi \cos\theta , \, -\sin\phi \sin\theta, \, \cos\phi ), \\
\hbtheta &= ( -\sin\theta, \, \cos\theta, \, 0 ), \\
\hbn &= ( \cos\phi \cos\theta, \, \cos\phi \sin\theta, \, \sin\phi ),
\end{split}
\end{equation}
where $(x,y,z)$ denotes a Cartesian vector in $\bbR^3$.  It is then
straightforward to show that on $\partial D$, a basis for the
two-dimensional harmonic vector fields is
\begin{equation}
\begin{split}
\bj_{H_1}(\theta,\phi) &= \left( \frac{-\sin\theta}{2+\cos\phi}, \,
\frac{\cos\theta}{2+\cos\phi}, \, 0 \right), \\
\bj_{H_2}(\theta,\phi) &= \left( \frac{-\sin\phi\cos\theta}{2+\cos\phi}, \,
\frac{-\sin\phi\sin\theta}{2+\cos\phi}, \, \frac{\cos\phi}{2+\cos\phi} \right).
\end{split}
\end{equation}
In fact, $\bj_{H_2} = \star_2\bj_{H_1}$ ($= \hbn \times \bj_{H_1}$ in 
vector notation).
Therefore, the linear combination which satisfies the requirement
\begin{equation}
\star_2\bj_H= -i\bj_H
\end{equation}
is given by
\begin{equation}
\bj_H = \bj_{H_1} - i \bj_{H_2}.
\end{equation}
In order to avoid redundancy, we omit a discussion of the
discretization of the boundary $\partial D$ and the following integral
operators; we only point out that the surface unknowns (and
consequently the Beltrami fields) are represented by their Fourier
series in $\theta$, and the resulting 2D boundary integral equations
are discretized in $\phi$ using a $50$-point, $16^{\text{th}}$-order
hybrid Gaussian-trapezoidal rule due to Alpert
\cite{alpert,martinsson}. See Section~6 of \cite{EpGr2} for more
details, as well as a forthcoming paper.  To summarize, we wish to
find frequencies and fields, $k$ and $\bxi$, respectively, which are
non-trivial solutions to
\begin{equation}\label{eq-linsys}
\begin{aligned}
\cB\bxi &= k \, \bxi &\quad &\text{in } D, \\
  \int\limits_{A} \bxi &= 0, &\quad & \\
i_n \bxi &= 0 &\quad &\text{on } \partial D,
\end{aligned}
\end{equation}
where the cycle $A$ is given by equation~(\ref{eq-torus}) with 
$\phi \in [0,2\pi)$ and $\theta=0$.

In fact, it is easy to see that the second condition 
in~(\ref{eq-linsys}) is only non-trivial for
solutions that are purely axially-symmetric. Otherwise, if
$\bxi$ is a closed 1-form on the boundary with
\begin{equation}\label{eq-rep}
\bxi = \left( a(\phi) d\theta + b(\phi) d\phi \right) e^{im\theta},
\end{equation}
where $m \neq 0$, then $im b(\phi)=a'(\phi)$ and therefore
\begin{equation}
\bxi=\frac{1}{im}d(a(\phi)e^{im\theta}).
\end{equation}

To this end, suppressing the separation of variables in the $\theta$
variable, we represent the unknown field $\bxi$ as
\begin{equation}
\bxi = i k \btheta - d^*\Psi + d\balpha,
\end{equation}
where $\btheta$, $\Psi$, and $\balpha$ are all assumed to depend
on the unknown surface charge $q$ using relations (\ref{srfint2}), 
(\ref{eqn3.28}), (\ref{eqn3.28a}), and (\ref{eqn3.29}). Note that this
representation is analogous to that of the magnetic field used earlier
in~(\ref{eqn29}).
Using this representation, the resulting linear system for the unknown 
$q$ is given by
\begin{gather}\label{eq-b}
\frac{q}{2} + i k \, K_{2,n}[\bm] - K_0[q] + K_3[\bj] = 0, \notag \\
\int\limits_{A} \bxi[q] = 0.
\end{gather}
Let the discretization of the above linear system be denoted as
\begin{equation}
\mathcal A \, q = 0.
\end{equation}
When $k$ is a resonant frequency, the matrix $\mathcal A$ is
singular.  When $k$ is \emph{near} a resonant frequency, $\mathcal
A$ is numerically ill-conditioned. If $k$ is allowed to be complex
(even though we showed earlier that all Beltrami resonances are
real), then one can use  Muller's method for finding  roots, 
see~\cite{muller}, to find a numerical zero of the function
\begin{equation}
f(k) = \frac{1}{r_1 \mathcal A^{-1} r_2},
\end{equation}
where $r_1$ and $r_2$ are \emph{fixed} random vectors with unit $\ell_2$ norm.
The matrix $\mathcal A^{-1}$ can be computed via Gaussian
elimination.  As $\mathcal A$ becomes increasingly ill-conditioned,
the proxy function $f$ approaches zero. Using this approach, the first
several non-zero Beltrami resonances in Fourier modes $0$, $1$, and
$2$ are calculated and given in Table~\ref{tab-ks}. As a
function of the wave-number $k$, the condition number of the matrix
$\mathcal A$,
\[
\kappa = \frac{\sigma_{\text{max}}}{\sigma_{\text{min}}},
\]
where $\sigma_{\text{max}}$ is the largest singular value and
$\sigma_{\text{min}}$ is the smallest singular value, is plotted in
Figure~\ref{fig-ks}. Spikes in the condition number correspond to
values of $k$ for which the linear system is singular,
i.e. Beltrami resonances.

In order to compute a Beltrami field in the volume once a resonant $k$
has been found, the null-vector, $q_0$, of $\mathcal A$ must be
calculated.  This can be done via Gaussian elimination or the
singular-value decomposition.  Once $q_0$ is computed, the Beltrami
field may be evaluated in the volume $D$, off of $\partial D$, via
smooth layer potential integrals using
representation~(\ref{eq-rep}). Using this method for computing the
null-vector, Figure~\ref{fig-volume} shows a Beltrami vector field in
the volume corresponding to the lowest resonance in the axisymmetric
$m=0$ mode. Figure~\ref{fig-slice} shows two-dimensional projections
in the $xz$-plane of the Beltrami fields corresponding to the lowest
resonance in the $m=0$, $m=1$, and $m=2$ modes, respectively. The
component of the vector field in the $\hbtheta$ direction can be
ignored since it necessarily has zero contribution normal to $\partial
D$.

\begin{table}[b]
\centering
\begin{tabular}{c|ccc}
& Resonance $k$ & $|f(k)|$ & $\text{dim} \, \text{Null}(\mathcal A)$\\
\hline
       & $3.6507029 + i \, 1.9\text{E-12}$ & $1.3\text{E-14}$ & $1$ \\
Mode 0 & $3.8577285 - i \, 1.3\text{E-12}$ & $7.5\text{E-15}$ & $1$ \\
       & $4.0272565 + i \, 2.3\text{E-14}$ & $4.3\text{E-14}$ & $1$ \\
\hline
       & $3.3392108 + i \, 3.2\text{E-13}$ & $2.8\text{E-14}$ & $1$ \\
Mode 1 & $3.8476760 + i \, 2.1\text{E-12}$ & $3.5\text{E-15}$ & $1$ \\
       & $4.3727343 + i \, 1.8\text{E-14}$ & $3.4\text{E-14}$ & $1$ \\
\hline
       & $3.1108713 - i \, 1.1\text{E-13}$ & $2.1\text{E-14}$ & $1$ \\
Mode 2 & $3.9513064 + i \, 1.7\text{E-12}$ & $4.6\text{E-15}$ & $1$ \\
       & $4.7182683 + i \, 2.6\text{E-12}$ & $1.5\text{E-14}$ & $1$ \\
\hline
\end{tabular}
\caption{The first several non-zero Beltrami resonances in the first
  few Fourier modes for the linear system (\ref{eq-b}) on the torus
  (\ref{eq-torus}).  Each value of $k$ was computed using Muller's
  rooting finding algorithm. As this boundary value problem is
  self-adjoint it has only real resonances. The imaginary parts in the
  table, which are $\mathcal O(10^{-14}),$ are zero to within the accuracy of
  this computation.  Higher resonances may or may not correspond to
  one-dimensional eigenspaces.}
\label{tab-ks}
\end{table}

\begin{figure}[p]
\centering
  \subfigure[Mode 0.]
  {\includegraphics[width=.45\linewidth]{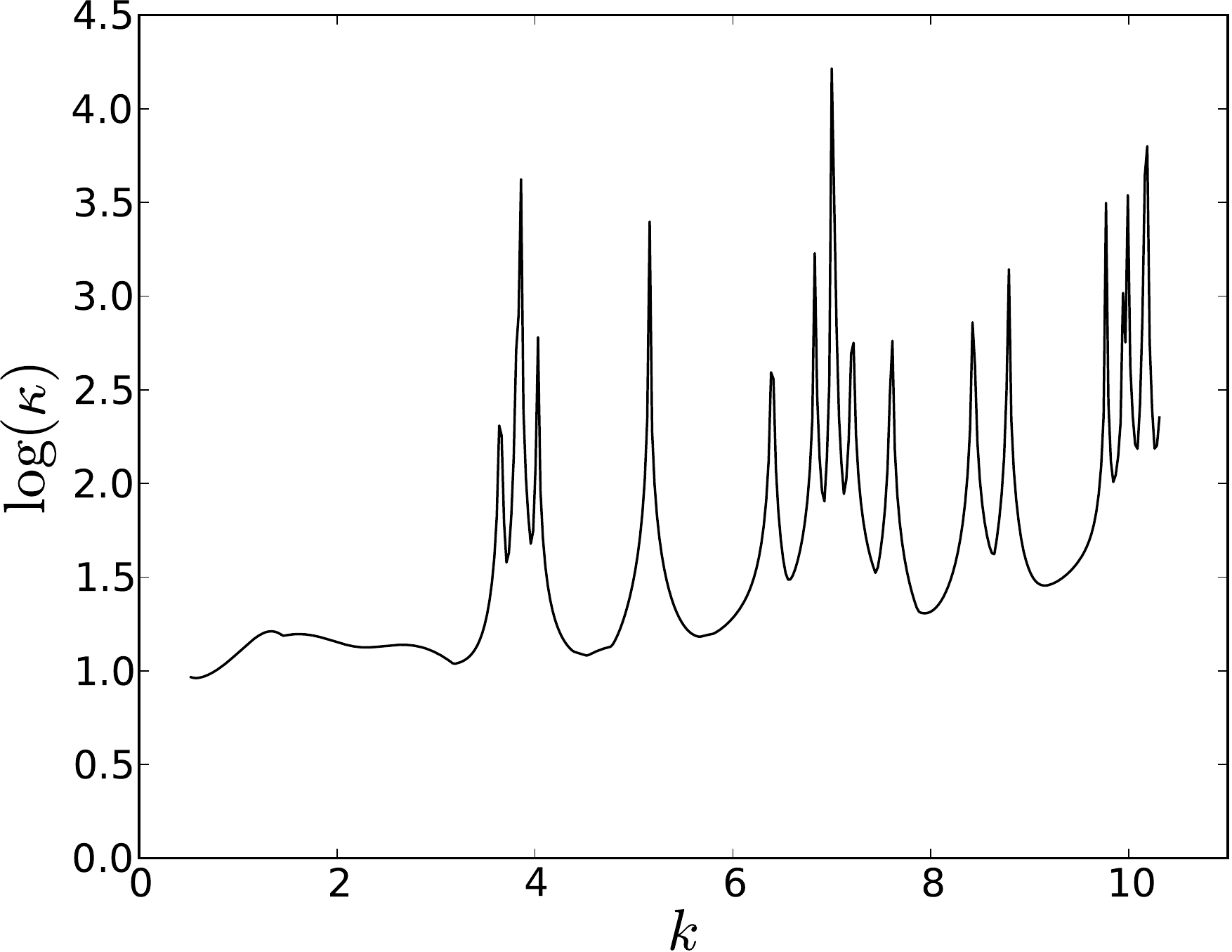}} \\
  \subfigure[Mode 1.]
  {\includegraphics[width=.45\linewidth]{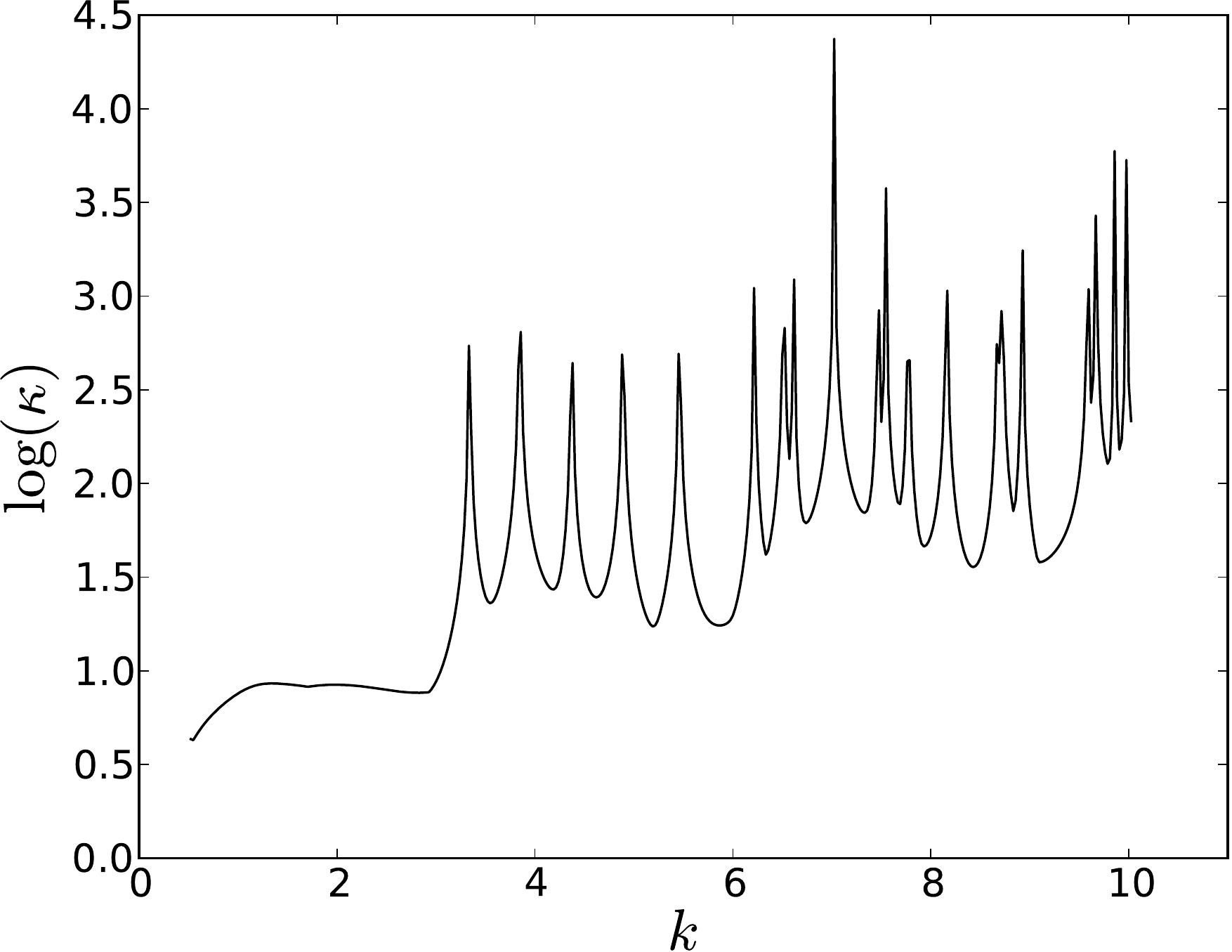}} \\
  \subfigure[Mode 2.]
  {\includegraphics[width=.45\linewidth]{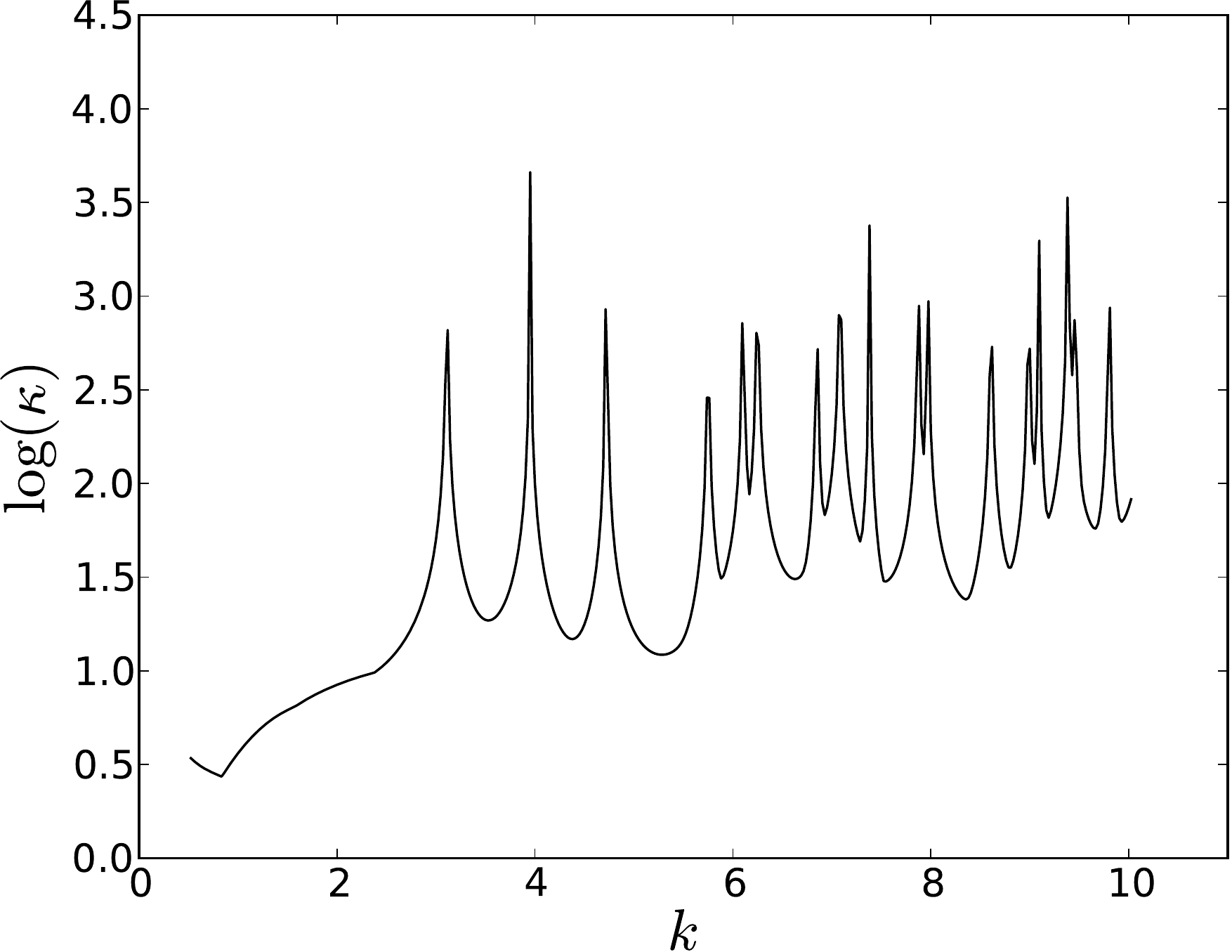}} \\
\caption{A plot of the condition number of the linear system
(\ref{eq-b}) on the torus (\ref{eq-torus}). Peaks
correspond to likely Beltrami resonances.}
\label{fig-ks}
\end{figure}

\begin{figure}[ht]
\centering
\includegraphics[width=.9\linewidth]{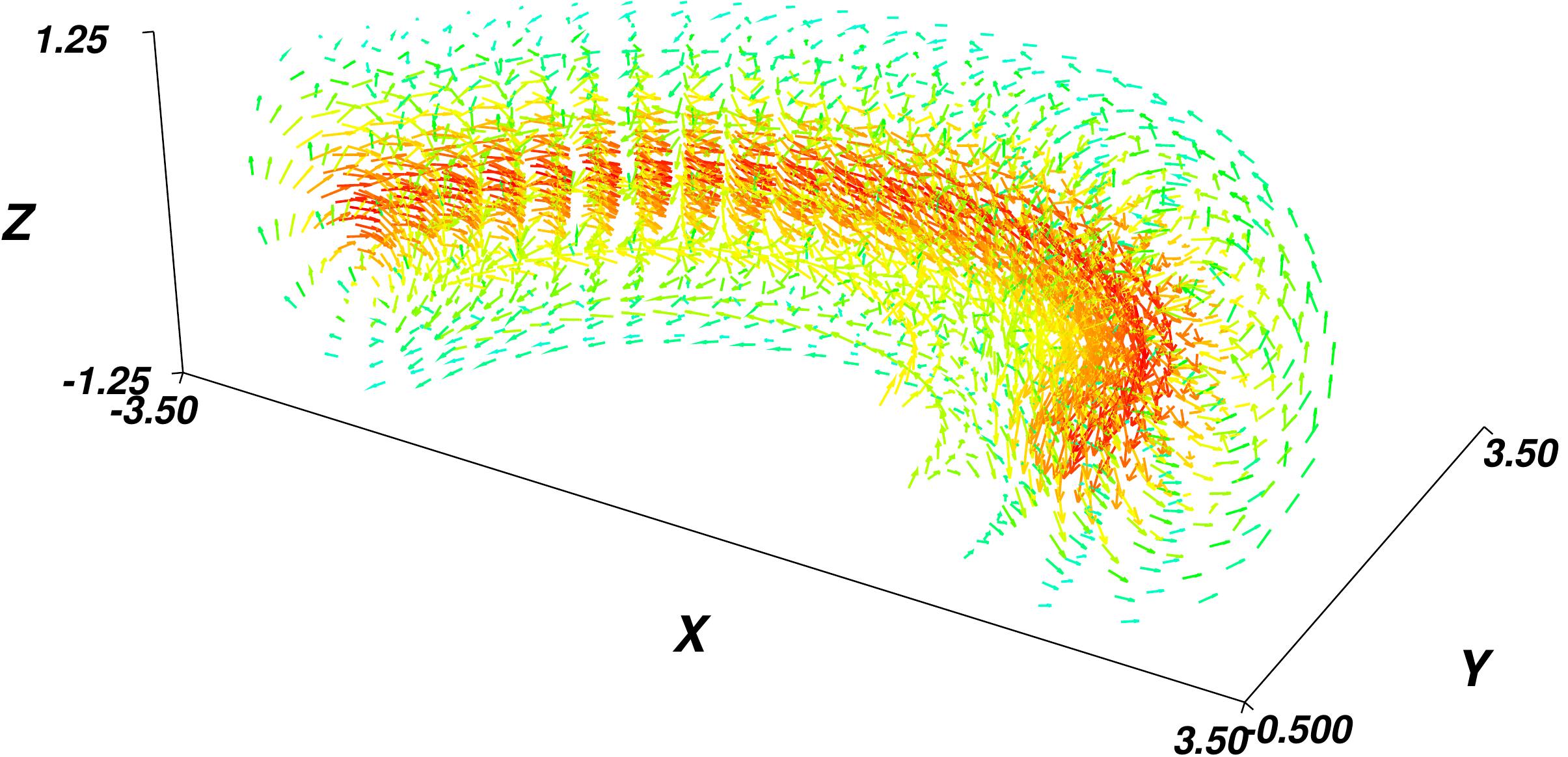}
\caption{A volume plot of the Beltrami field inside the
torus corresponding to the lowest resonance in the purely axisymmetric
mode.}
\label{fig-volume}
\end{figure}

\begin{figure}[p]
\centering
  \subfigure[Mode 0.]
  {\includegraphics[width=.45\linewidth]{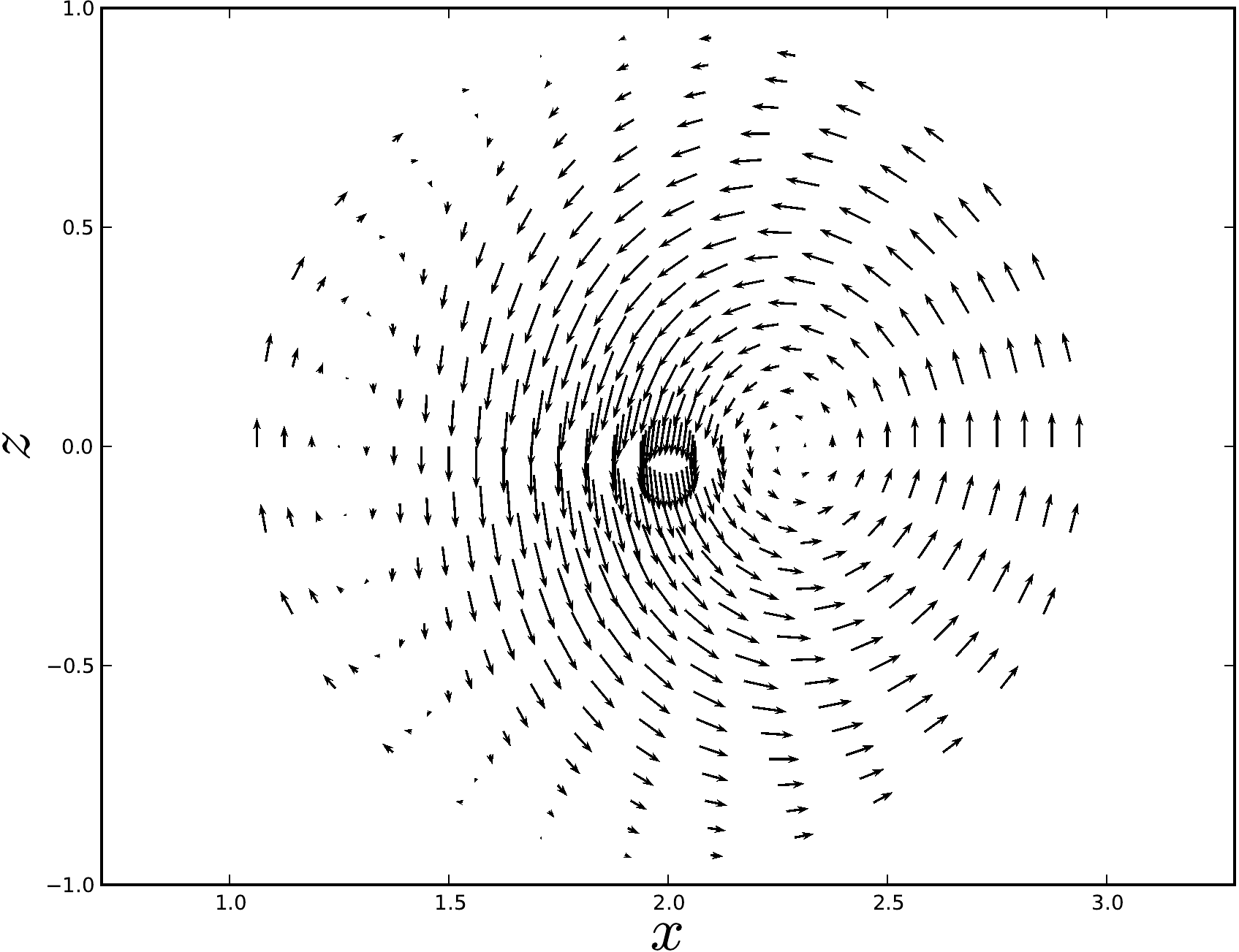}} \\
  \subfigure[Mode 1.]
  {\includegraphics[width=.45\linewidth]{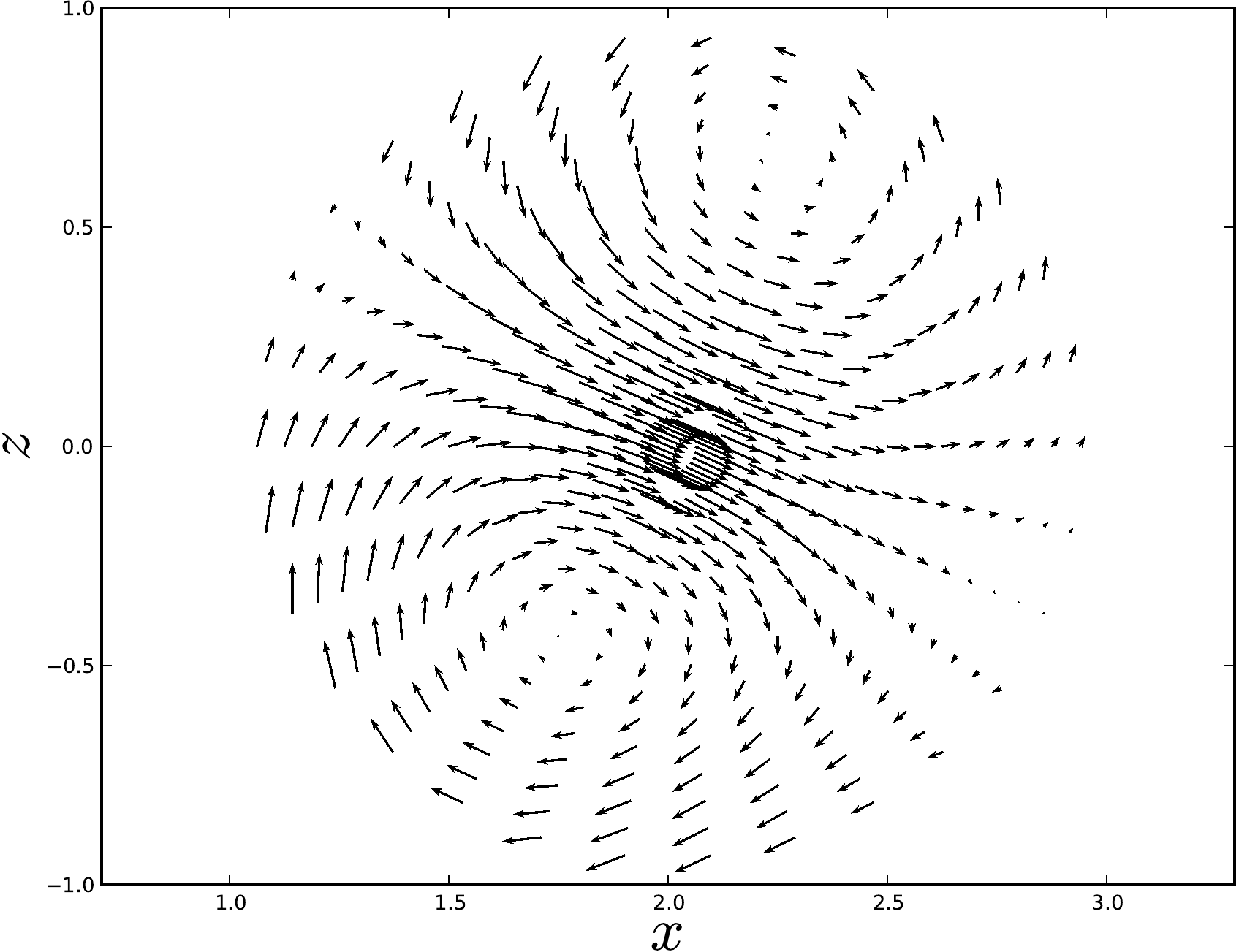}} \\
  \subfigure[Mode 2.]
  {\includegraphics[width=.45\linewidth]{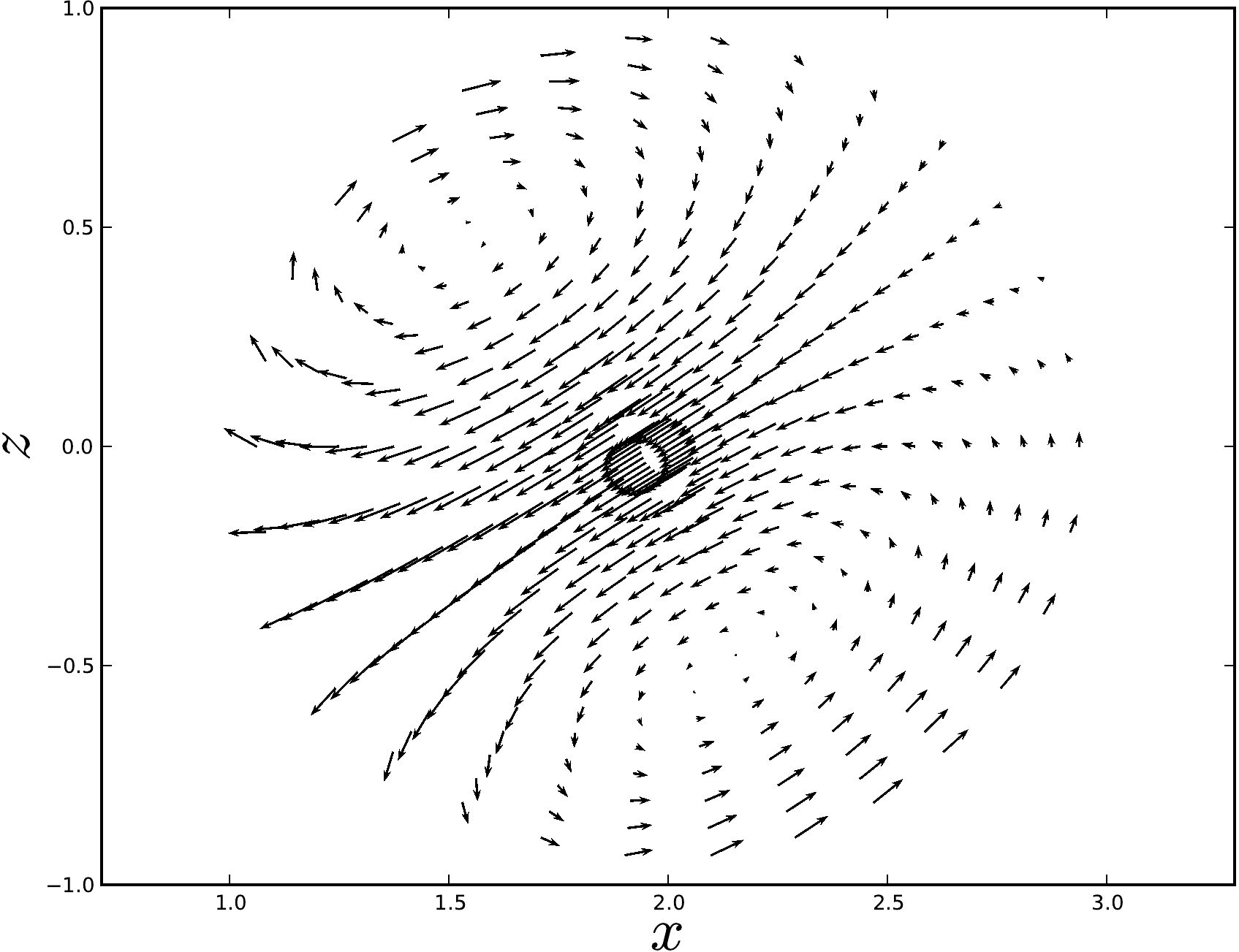}} \\
\caption{Projections in the $xz$-plane of Beltrami fields
corresponding to the lowest resonance in the first few
Fourier modes.}
\label{fig-slice}
\end{figure}

Lastly, the final set of numerical experiments we conduct addresses
the relationship between the Beltrami resonances $k_j$ and a
continuous family of boundary value problems. One may replace the
integral condition in the boundary value problem~(\ref{eq-linsys})
with a family of homogeneous topological conditions:
\begin{equation}\label{eq-tau}
  t \int\limits_{A} \bxi + (1-t) \int\limits_{B} \bxi = 0,
\end{equation}
with $t \in [0,1]$ and $B$ given by equation~(\ref{eq-torus}) with $\theta \in
[0,2\pi)$ and $\phi=\pi$. We wish to study the dependence of the location of
the Beltrami resonances in the axisymmetric mode as a function of $t$.  A plot
of the condition number of $\mathcal A$ for various values of $k$ and $t$ is
given in Figure~\ref{fig-tau}. Notice that as $t$ moves from $0 \rightarrow 1$,
the location of the first resonance decreases. For $t\neq 0,$ the null-space is
spanned by $\{du:\: u\text{ is harmonic}\}.$ At $t = 1$ the generator of
$\cH^1(D)$ defines an additional one-dimensional null-space since its integral
around the $A$-cycle vanishes. In the notation of the previous section, this
cycle defines $\lambda_D.$ It is interesting to note that while most of the
resonances appear to move as the value of $t$ changes, the resonance near
$3.8577285...$ seems to be fixed (this resonance appears in
Table~\ref{tab-ks}).  The nature of this behavior has yet to be understood.

\begin{figure}[th]
\centering
\includegraphics[width=.8\linewidth]{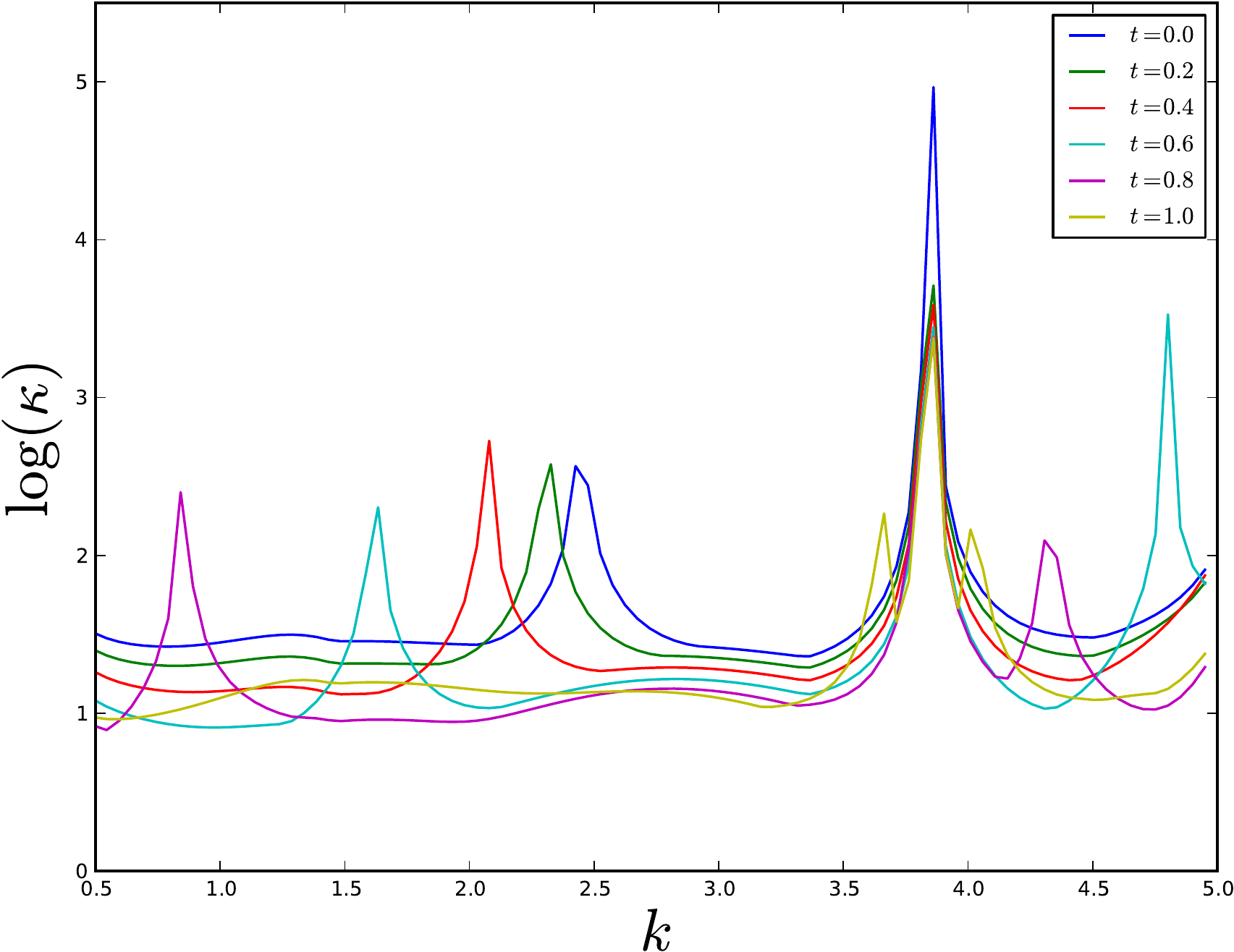}
\caption{The condition number of the matrix $\mathcal A$ for various values of
  $k$ and $t$ (different boundary conditions) in the axially symmetric mode,
  see equation~(\ref{eq-tau}). As described earlier, peaks correspond to likely
  Beltrami resonances.}
\label{fig-tau}
\end{figure}

Using an analogous code, calculations similar to the previous
examples are possible for any genus~$1$ surface of revolution. Results
of those calculations, as well as results for genus~$>1$ surfaces of
revolution, will be presented in a subsequent paper.

\subsection{The Time-Harmonic Maxwell Equation}\label{subsec7.4}
A similar analysis can be applied to the time-harmonic Maxwell
equations.  We use the usual inner product:
\begin{equation}
  \langle(\bxi,\bEta),(\balpha,\bBeta)\rangle=
\int\limits_{D}[\bxi\wedge\star_3\overline{\balpha}+
\bEta\wedge\star_3\overline{\bBeta}].
\end{equation}
Let $\cE_2(D)$ be the $L^2$-closure of divergence-free pairs,
$(\bxi,\bEta)$.  Working with smooth forms, we see that
\begin{equation}\label{eqn3.32.002}
   \langle\cL(\bxi,\bEta),(\balpha,\bBeta)\rangle-
   \langle(\bxi,\bEta),\cL(\balpha,\bBeta)\rangle=
i\int\limits_{\partial D}[\star_3\bEta\wedge\overline{\balpha}-\bxi\wedge\star_3\overline{\bBeta}].
\end{equation}
The integral on the right hand side defines a symplectic structure on
the space of tangential boundary data
$(\bxi_t,[\star_3\bEta]_t)$. From this it is clear that the problem of
finding domains on which the operator $\cL$ is essentially
self-adjoint reduces to the problem of finding Lagrangian subspaces
relative to this symplectic form.

The classical local boundary conditions, either $\bxi_t=0$, or
$[\star_3\bEta]_t=0$, certainly define such Lagrangian subspaces. Using the
Hodge decomposition on $\partial D$, formula~\eqref{eqn7.2.005},
gives a very large and interesting family of non-local boundary
conditions that define unbounded, self-adjoint operators on
$\cE_2(D)$.

As noted, the integral on the right hand side of~\eqref{eqn3.32.002}
defines a symplectic form on the product space $H^1_{\dR}(\partial
D)\times H^1_{\dR}(\partial D)$. For each Lagrangian subspace
$\lambda\in \Lambda_H^1(\partial D)$, the subspace
$\lambda\oplus\lambda$ is a Lagrangian subspace of the product space,
but there are many others. Let $\Lambda^{(2)}(\partial D)$ denote the
Lagrangian subspaces of $H^1_{\dR}(\partial D)\times
H^1_{\dR}(\partial D)$. For $\lambda^{(2)}\in \Lambda^{(2)}(\partial
D)$, let $\cD(\lambda^{(2)})$ denote $(\bxi,\bEta)\in\cE_2(D)$ so that
$d\bxi$ and $d^*\bEta$ are in $L^2(D)$,
\begin{equation}
  d_{\partial D}\bxi_t=0 \quad
\text{ and } \quad d_{\partial D}[\star_3\bEta]_t=0,
\end{equation}
and the pair of cohomology classes defined by $([\bxi_t],
[\star_3\bEta]_t)$ lies in $\lambda^{(2)}$. In terms of the Hodge
decomposition, this can be rephrased as
\begin{equation}
  \bxi_t=df+\omega_1 \quad \text{ and } \quad [\star_3\bEta]_t=dg+\omega_2
\end{equation}
where $([\omega_1],[\omega_2])\in \lambda^{(2)}$. It is important to
note that this condition makes sense even for distributional solutions
of the equation $d_{\partial D}\bomega_t=0$. It is quite clear that,
for two pairs $(\bxi,\bEta)$ and $(\balpha,\bBeta)$ in
$\cD(\lambda^{(2)})$, the right hand side of~\eqref{eqn3.32.002}
vanishes. This shows that $\cL$ with domain $\cD(\lambda^{(2)})$ takes
values in $\cE_2(D)$ and is symmetric.

As in the previous case, a complicating feature of these operators is that each
has an infinite dimensional null-space. Let $(u,v)$ be a pair of $W^{1,2}(D)$
harmonic functions defined in $D$, then the pair of divergence-free 1-forms
$(du,\star_3dv)$ belongs to $\cD(\lambda^{(2)})$ for every
$\lambda^{(2)}\in\Lambda^{(2)}_H(\partial D)$, and clearly
$\cL(du,\star_3dv)=0$. Indeed this set constitutes a finite codimension subspace of
the null-space of $(\cL,\cD(\lambda^{(2)}))$.

The self-adjointness of these operators is again not entirely obvious.  Though
we do not give a detailed proof here, Lemmas~\ref{lem7.2.008} and~\ref{lem7.5}
can be used to give a proof (quite similar to the proof of
Theorem~\ref{thm7.2.001}) of the following theorem.
\begin{theorem}\label{thm7.9.0001} 
Let $D$ be a connected bounded domain in $\bbR^3$ with a smooth
boundary. For each $\lambda^{(2)}\in\Lambda^{(2)}_H(\partial D)$, the
densely operator operator $(\cL,\cD(\lambda^{(2)}))$ is self-adjoint. The self-adjoint
operator defined on the orthocomplement of the
$\Ker(\cL,\cD(\lambda^{(2)}))$ has a compact resolvent, and therefore
a real, discrete spectrum, $\{k_j(\lambda^{(2)})\}$. We let
$\{(\bxi_j(\lambda^{(2)}),\bEta_j(\lambda^{(2)}))\}$ denote the
corresponding eigenvectors. They satisfy
\begin{equation}
\begin{split}
  d\bxi_j(\lambda^{(2)})&=ik_j(\lambda^{(2)})\bEta_j(\lambda^{(2)}), \\
d^*\bEta_j(\lambda^{(2)})&=-ik_j(\lambda^{(2)})\bxi_j(\lambda^{(2)}),
\end{split}
\end{equation}
and
\begin{equation}
  d_{\partial D}\bxi_{jt}(\lambda^{(2)})=0 \quad \text{ and }
\quad d_{\partial D}[\star_3\bEta_j(\lambda^{(2)})]_t=0,
\end{equation}
which, as $k_j(\lambda^{(2)})\neq 0$,   implies that
\begin{equation}
  i_{\bn}\bxi_{j}(\lambda^{(2)})=0 \quad \text{ and } \quad 
i_{\bn}\star_3\bEta_j(\lambda^{(2)})=0.
\end{equation}
\end{theorem}

The final statement in the theorem implies that the eigenvectors of
these operators are also $k_j$-Neumann fields. These eigenfields
satisfy an additional condition:
\begin{equation}
  \int\limits_{\partial
    D}\bxi_{jt}(\lambda^{(2)})\wedge\overline{[\star_3\bEta_j(\lambda^{(2)})]_t}=0.
\end{equation}
It should be noted that in the exterior case we proved that if
this integral vanishes for an outgoing field, then it is identically zero.
Below we explain why this condition does not hold for every
$k$-Neumannn field defined in $D$.

There are several distinguished elements in $\Lambda^{(2)}(\partial
D)$, which merit special attention. If $\lambda_D$ is the element
defined by $H^1_{\dR}(D)\hookrightarrow H^1_{\dR}(\partial D)$, then
$\lambda_D\oplus\lambda_D$ is the analogue of the element defined in
the previous subsection. We can equally well consider the subspace
$\lambda_{D^c}\oplus\lambda_{D^c}$. The eigenfields of
$(\cL,\cD(\lambda_D\oplus\lambda_D))$ are $k$-Neumann fields
$(\bxi,\bEta)$ with
\begin{equation}
\begin{split}
  \int\limits_{A_j}\bxi&=ik\int\limits_{S_j}\bEta=0,\\
  \int\limits_{A_j}\star_3\bEta&=-ik\int\limits_{S_j}\star_3\bxi=0,
\end{split}
\end{equation}
for $j=1,\dots,p$. These are {\em zero-flux} $k$-Neumann fields. The
existence of these fields is proved in~\cite{Kress1}.

Let ${B_j}$ denote a basis of 1-cycles on $\partial D$ that bound
2-cycles in $D^c$. A closed 1-form, $\balpha$, belongs to
$\lambda_{D^c}$ if and only if for $j = 1, \ldots, p$
\begin{equation}
  \int\limits_{B_j}\balpha=0.
\end{equation}
From the boundary value problem
$(\cL,\cD(\lambda_{D^c}\oplus\lambda_{D^c}))$ we get a sequence of
frequencies for which there are $k$-Neumann fields $(\bxi,\bEta)$
defined in $D$ that satisfy:
\begin{equation}
  \int\limits_{B_j}\bxi=\int\limits_{B_j}\star_3\bEta=0.
\end{equation}

We observe that if $k$ is not in the spectrum of either
$(\cL,\cD(\lambda_D\oplus\lambda_D))$ or
$(\cL,\cD(\lambda_{D^c}\oplus\lambda_{D^c}))$, then we can find
$k$-Neumann fields $(\bxi,\bEta)$ for which
\begin{equation}
  \int\limits_{\partial D}\bxi\wedge\star_3\overline{\bEta}\neq 0.
\end{equation}
With this hypothesis on $k$ we can find $k$-Neumann fields
$(\balpha_1,\bBeta_1)\in \cD(\lambda_D\oplus\lambda_D)$ so that for
$j = 1, \ldots, p$ the integrals
\begin{equation}
  \int\limits_{A_j}\balpha_1 \quad \text{ and } \quad \int\limits_{A_j}\bBeta_1
\end{equation}
take arbitrarily specified values, and $(\balpha_2,\bBeta_2)\in
\cD(\lambda_{D^c}\oplus\lambda_{D^c})$ so that the integrals
\begin{equation}
  \int\limits_{B_j}\balpha_2 \quad \text{ and } \quad \int\limits_{B_j}\bBeta_2
\end{equation}
also take arbitrarily specified values. Since $\lambda_D$ and
$\lambda_{D^c}$ are complementary Lagrangian sub-spaces, these values
can be selected so that
\begin{equation}
  \int\limits_{\partial D}(\balpha_1+\balpha_2)\wedge\star_3(\overline{\bBeta_1+\bBeta_2})=
\int\limits_{\partial D}[\balpha_1\wedge\star_3\overline{\bBeta_2}+\balpha_2\wedge\star_3\overline{\bBeta_1}]
\end{equation}
is non-zero.

For the study of $k$-Neumann fields, the elements
\[
\lambda^{(2)}_1=\{0\}\times H^1_{\dR}(\partial D) \quad 
\text{ and } \quad \lambda^{(2)}_2=H^1_{\dR}(\partial D)\times\{0\}
\]
are perhaps of greater interest.  Recall that $k\in E^-_n$ if there
exists $(\bxi,\bEta)$ which is a topologically trivial $k$-Neumann
field, that is, $[\bxi_t]$ is trivial in $H^1_{\dR}(\partial D)$. The
eigenfields of the operator $(\cL,\cD(\lambda_1^{(2)}))$ are precisely
the topologically trivial $k$-Neumann
fields. Theorem~\ref{thm7.9.0001} shows that the operator
$(\cL,\cD(\lambda_1^{(2)}))$ is self-adjoint, with compact resolvent on
the orthogonal complement of the null-space. The existence of infinite
sequence of real frequencies $\{k_j(\lambda_1^{(2)})\}$ for which
topologically trivial $k$-Neumann fields exist is immediate. The
spectral theory of the operator $(\cL,\cD(\lambda_2^{(2)}))$ leads to
the existence of a sequence of frequencies for which there are
$k$-Neumann fields with the $\bEta$-component topologically
trivial. Note that if there is a $k_0$ for which there is a
$k$-Neumann field with both $\bxi_t$ and $[\star_3\bEta]_t$
topologically trivial, then $k_0$ belongs to the spectrum of the
operator $(\cL,\cD(\lambda^{(2)}))$ for every
$\lambda^{(2)}\in\Lambda^{(2)}_H(D)$.

For $\lambda^{(2)}\in\Lambda^{(2)}_H(\partial D)$, we let
$\sigma(\cL,\lambda^{(2)})$ denote the non-zero spectrum of
$(\cL,\cD(\lambda^{(2)}))$. We have the following analogue of
Corollary~\ref{cor7.7.003}:
\begin{corollary}\label{cor7.10.001} Let
  \begin{equation}
    E_{\cL}=\bigcap_{\lambda^{(2)}\in \Lambda^{(2)}_H(\partial D)}\sigma(\cL,\lambda^{(2)}).
  \end{equation}
For $k\in \mathbb C^+\setminus E_{\cL}$, the space of $k$-Neumann fields, $\cN_k(D)$, has dimension $2p$.
\end{corollary}
\noindent
The proof of this corollary is essentially identical to that of
Corollary~\ref{cor7.7.003}.
\begin{remark} It is again  quite an interesting question whether or not
  $E_{\cL}=\emptyset$. As in the Beltrami case, if
  $E_{\cL}=\emptyset$, then $\dim \cN_k(D)=2p$ for every $k\in \mathbb
  C^+$. If $E_{\cL}\neq\emptyset$, then this set is a new spectral
  invariant of the embedding of $\partial D$ into $\bbR^3$.
\end{remark}

The study of the family of operators
$\{(\cL,\cD(\lambda^{(2)})):\:\lambda^{(2)}\in\Lambda^{(2)}_H(\partial
D)\}$ seems a very natural and interesting problem. In particular it
seems quite interesting to analyze how the spectrum depends on the
choice of Lagrangian subspace. The existence of this large family of
operators whose spectra include the $k$-Neumann resonances shows that
these rather mysterious numbers are really part of a larger whole. It
seems quite likely that any real $k$ arises as eigenvalue for one of
the operators $(\cL,\cD(\lambda^{(2)}))$.

\subsection{$k$-Neumann fields on the Sphere}

If $\partial D$ is simply connected, then there is a unique,
self-adjoint boundary value problem $(\cL,\cD(0))$.  In this case
\[
  \Ker(\cL,\cD(0))=\{ (du,\star_3dv):\: u\text{ and }v \text{ are harmonic functions}\}.
\]
This operator has a real discrete spectrum, from which it is apparent
that for some values of $k$, $\dim\cN_k(D)$ does not equal $2p$ (which
in this case is $0$). We conclude our discussion by constructing
$k$-Neumann fields on the unit ball $B_1$.

\begin{figure}[th]
\centering
  \subfigure[Tangent field defined by $Y^1_2$.]
  {\includegraphics[width=.3\linewidth]{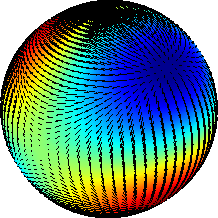}} \qquad \qquad
  \subfigure[Tangent field defined by $Y^2_4$.]
  {\includegraphics[width=.3\linewidth]{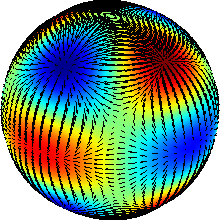}} \\
  \subfigure[Tangent field defined by $Y^6_{15}$.]
  {\includegraphics[width=.3\linewidth]{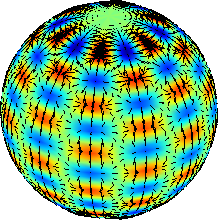}} \qquad \qquad
  \subfigure[Tangent field defined by $Y^{13}_{27}$.]
  {\includegraphics[width=.3\linewidth]{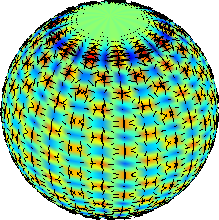}}
\caption{Tangent components of Beltrami fields on the unit ball
  defined by different choices of spherical harmonics.}
\label{SphFig}
\end{figure}

On the unit ball we
can construct $k$-Neumann fields from scalar Dirichlet eigenfunctions. Let
\begin{equation}
  L_{ij}=x_i\pa_{x_j}-x_j\pa_{x_i}
\end{equation}
denote the angular momentum operators. As is well known 
\begin{equation}
\Delta_{\bbR^3}L_{ij}=L_{ij}\Delta_{\bbR^3}.
\end{equation}
If $\Delta u=-k^2u$ and $u\restrictedto_{\|x\|=1}=0$, then the 1-form
\begin{equation}
\begin{split}
  \bxi_u&=L_{23}udx_1+L_{31}udx_2+L_{12}udx_3\\
&=\frac{1}{2}\star_3d(r^2du),
\end{split}
\end{equation}
is easily seen to be divergence-free and to vanish along the unit
sphere. If we let
\begin{equation}
  \bEta_u=\frac{1}{ik}d\bxi_u,
\end{equation}
then $d^*\bEta_u=-ik\bxi_u$. The 1-form $\bxi_u$ vanishes in all directions
along the unit sphere, which implies that $\bEta_u\restrictedto_{TS^2}=0$, and
therefore the pair $(\bxi_u,\bEta_u)$ is a $k$-Neumann field. Note that not
only is $\bxi_{ut}$ closed; it is identically zero. An easy calculation
verifies the general fact that $[\star_3\bEta_u]_t$ does not vanish. This shows
that, even though for most values of $k$ the $\dim \cN_{k}(B_1)=0$, among
$k$-Neumann resonances it can be arbitrarily large. From this example it does
not seem likely that for general domains $D$ the Dirichlet spectrum of the
Laplace operator on scalar functions will have any connection to interior
$k$-Neumann resonances.

As follows from the general theory,
$\cN_k(B_1)=\cN^{1,0}_k(B_1)\oplus\cN^{0,1}_k(B_1)$. Hence, from these
fields we can create Beltrami fields as $\bxi_u\pm i\star_3\bEta_u$. A
calculation shows that
\begin{equation}
\begin{split}
  \star_3\bEta_u&=\Delta u\bx\cdot d\bx-2du-\bx\cdot\pa_{\bx}du\\
&=-k^2 u\bx\cdot d\bx-2du-\bx\cdot\pa_{\bx}du.
\end{split}
\end{equation}
Eigenfunctions of the Laplace operator are of the form
$r^nf(k^2_{nj}r^2)Y(\theta,\phi)$, where $Y$ is a spherical harmonic
of degree $n$. The set of spherical harmonics is a $2n+1$-dimensional
vector space. The map $u\mapsto dr\wedge du$ is injective unless
$n=0$, in which case it is zero. It's quite clear that the subspace of
solutions of the form $(\bxi_u,\bEta_u)$ intersects those of the form
$(-\star_3\bEta_u,\star_3\bxi_u)$ only in the trivial solution. This
is because $\bxi_u$ vanishes along the sphere and $\star_3\bEta_u$
does not. This indicates that $\dim\cN_k(B_1)=2(2n+1)$, and therefore
\begin{equation}
  \dim\cN^{1,0}_k(B_1)= \dim\cN^{0,1}_k(B_1)=2n+1.
\end{equation}
As $\bxi_u$ vanishes along $bB_1$, we easily see that the tangent components
of $\bxi_u\pm i\star_3\bEta_u$ along $bB_1$ are constant multiples of the
vector field $\nabla_{S^2}Y$, the tangential gradient of $Y(\theta,\phi)$.
Several examples are shown in Figure~\ref{SphFig}.

\appendix

\section{Boundary operators}\label{bdryops}
Let $\Gamma$ be a smooth boundary of either a bounded domain $D$ or
unbounded domain $\Omega$ in $\bbR^3$.  We now describe the limits of
the tangential components of $\bxi$ and $\bEta$ along $\Gamma$. Along
$\Gamma$, we can introduce an adapted local basis of orthonormal
1-forms, $\{\omega_1,\omega_2,\nu\}$, and $\bn$ the outward unit
normal vector. In this basis we have
$\bxi^{\pm}=a_{\pm}\omega_1+b_{\pm}\omega_2+c_{\pm}\nu$. The
tangential part of $\bxi^{\pm}$ at $x\in\Gamma$ is the 1-form
$\bxi^{\pm}(x)$ restricted to directions tangent to $\Gamma$, i.e.,
$T_x\Gamma$. In terms of components along $\Gamma$ we identify the
tangential part with
\begin{equation}
\bxi^{\pm}_t=a_{\pm}\omega_1+b_{\pm}\omega_2.
\end{equation}
If the 2-form, $\bEta^{\pm}=e_{\pm}\omega_1\wedge\nu+
f_{\pm}\omega_2\wedge\nu+g_{\pm}\omega_1\wedge\omega_2$, then the
tangential components of $\bEta^{\pm}$ along $\Gamma$ are identified
with the tangential components of the one-form $i_{\bn}\bEta^{\pm}$
wedged with $-\nu:$
\begin{equation}
\bEta^{\pm}_t=e_{\pm}\omega_1\wedge\nu+f_{\pm}\omega_2\wedge\nu.
\end{equation}
The normal component of $\bxi$ is $i_{\bn}\bxi$ and that of $\bEta$ is
simply $\bEta\restrictedto_{\Gamma}$.
The boundary operators \cite{EpGr,EpGr2} are given by:
\begin{align*}
K_0[r](\bx) &= \int_\Gamma \frac{\partial g_k}{\partial n_{\bx}}(\bx - \by) \,
r(\by) \, dA(\by), \\
K_1[r](\bx_0) &= \star_2 d_{\Gamma} \int_\Gamma g_k(\bx - \by) \, 
r(\by) \, dA(\by) \, ;
\end{align*}
$K_0$ is an operator of order $-1$ and $K_1$ is an operator of order
$0$.  Furthermore,
\begin{align*}
K_{2,n} [\bj](\bx) &= i_{\bn}\left[\int_\Gamma  g_k(\bx - \by) \, 
\bj(\by)\cdot d\bx \, dA(\by)\right] \, , \\
K_{2,t} [\bj](\bx) &= \star_2\left[\int_\Gamma   g_k(\bx - \by) \, 
\bj(\by)\cdot d\bx \, dA(\by)\right]_{t} \, .
\end{align*}
These are operators of order $-1$. Finally, the operators $K_3$ of
order $0$ and $K_4$ of order $-1$ are given as:
\begin{align*}
K_3[\bj](\bx) &= \int_\Gamma 
 d_{\bx} g_k(\bx - \by) \cdot (\bj(\by) \times \bn(\bx))   
\, dA(\by)  \, , \\
K_4[\bj](\bx) &= \int_\Gamma 
\left[ d_{\bx} g_k(\bx - \by) \, (\bj(\by) \cdot \bn(\bx)) -  
\frac{\partial g_k}{\partial n_{\bx}} (\bx - \by) \, \bj(\by)\cdot d\bx \right] 
\, dA(\by) \, .
\end{align*}

\frenchspacing
{\small \bibliographystyle{siam}{\bibliography{alla-k}}}

\end{document}